\documentclass[a4paper,two-side,12pt]{article}
\usepackage{mathtools}
\usepackage[english]{babel}
\usepackage{fancyhdr}
\usepackage[multiple]{footmisc}
\usepackage[latin1]{inputenc}\usepackage[multiple]{footmisc}
\usepackage[T1]{fontenc}
\usepackage{amsmath,amssymb,amsfonts}
\usepackage{amsthm}
\usepackage{latexsym}
\usepackage{amssymb}

\usepackage{calc}
\usepackage{multicol}

\usepackage{hyperref}
\newtheorem{thm}{Theorem}[section]
\newtheorem{prop}{Proposition}[section]

\newtheorem{lem}{Lemma}[section]

\newcommand{\R}{\mathbb{R}}
\numberwithin{equation}{section}
\newcommand{\N}{\mathbb{N}}

\newcommand{\bb}{B\!B\!l}

\title{Non-stability of Paneitz-Branson type equations in arbitrary dimensions.}
\author{Laurent Bakri \footnote{Departamento de Matem\'atica, Universidad T\'ecnica Federico Santa Mar\'ia, 1680 Av. espa\~na
Valpara\'iso, Chile. E-mail :laurent.bakri@gmail.com.} \footnote{The first author was supported by PROYECTO BASAL PFB03 CMM, Universidad de Chile, Santiago (Chile).}
\and  Jean-Baptiste Casteras \footnote{UFRGS, Instituto de Matem\'atica, Av. Bento Goncalves 9500, 91540-000 Porto
Alegre-RS, Brasil Phone : (55) 51 3308-6208. E-mail Jean-Baptiste.Casteras@univ-brest.fr} \footnote{The second author was supported by the CNPq (Brazil) project 501559/2012-4. }}

\date{}
\begin{document}
\maketitle
\begin{abstract}
Let $(M,g)$ be a compact riemannian manifold of dimension $n\geq 5$. We consider a Paneitz-Branson type equation with general coefficients 
\begin{equation}
\label{abstr}\tag{E}
\Delta_g^2 u-div_g (A_g du) +h u=|u|^{2^\ast-2-\varepsilon}u\ on\ M,
\end{equation}
where $A_g$ is a smooth symmetric $(2,0)$-tensor, $h\in C^\infty (M)$, $2^\ast =\dfrac{2n}{n-4}$ and $\varepsilon$ is a small positive parameter. Assuming that there exists a positive nondegenerate solution of \eqref{abstr} when $\varepsilon=0$ and under suitable conditions, we construct solutions $u_\varepsilon$ of type $(u_0-\bb_\varepsilon)$ to \eqref{abstr} which blow up at one point of the manifold when $\varepsilon$ tends to $0$ for all dimensions $n\geq 5$. \vspace{0,3cm}
\end{abstract}

\phantom{}\newline
\section{Introduction and statements of the results}
Let $(M,g)$ be a compact riemannian manifold of dimension $n\geq 5$. We will be interested in solutions $u\in C^{4,\theta}(M)$, $\theta \in (0,1)$, of the following equation
 \begin{equation}
\label{eqin}
P_g u:= \Delta_g^2 u-div_g (A_g du) +h u=|u|^{2^\ast-2}u,
\end{equation}
where $A_g$ is a smooth symmetric $(2,0)$-tensor, $h\in C^\infty (M)$ and $2^\ast =\dfrac{2n}{n-4}$. Following the terminology introduced in \cite{djadli2000paneitz}, the operator $P_g$ has been referred to as a Paneitz-Branson type operator with general coefficients. When $A_g$ is given by
\begin{equation}
\label{paneitz}
A_g= A_{paneitz}:=\dfrac{(n-2)^2+4}{2(n-1)(n-2)}R_g g -\dfrac{4}{n-2}Ric_g,
\end{equation}
where $R_g$ (resp. $Ric_g$) stands for the scalar curvature (resp. Ricci curvature) with respect to the metric $g$, and $h=\dfrac{n-4}{2}Q_g$ where $Q_g$ is the $Q$-curvature with respect to the metric $g$ which is defined by
$$Q_g=\dfrac{1}{2(n-1)}\Delta_g R_g+\dfrac{n^3 -4n^2+16n - 16}{8(n-1)^2(n-2)^2}R_g^2-\dfrac{2}{(n-2)^2}|Ric_g|_g^2,$$
then $P_g$ is the so-called Paneitz-Branson operator and equation \eqref{eqin} is referred to as the Paneitz-Branson equation. 
It is well known that the Paneitz operator is conformally invariant, i.e. if $\tilde{g}=\varphi^{\frac{4}{n-4}}g$ then, for all $u\in C^\infty (M)$, we have
$$P_g^n (u\varphi)= \varphi^{\frac{n+4}{n-4}}P_{\tilde{g}}^n (u).$$
We also point out that if $(M,g)$ is Einstein ($Ric_g =\lambda g$, $\lambda\in \R$), then the Paneitz-Branson operator takes the form
\begin{equation}
\label{paneitzcons}
P_g u =\Delta^2_g u +b\Delta_g u +cu,
\end{equation}
where $b=\dfrac{n^2-2n-4}{2(n-1)}\lambda$ and $c=\dfrac{n(n-4)(n^2-4)}{16(n-1)^2}\lambda$. More generally, when $b$ and $c$ are two real numbers, the operator $P_g$ defined in \eqref{paneitzcons} is referred to as a Paneitz-Branson type operator with constant coefficients. 
Existence, compactness and stability of solutions to \eqref{eqin} when $P_g$ is a Paneitz-Branson type operator with constant coefficients, have been widely investigated this last decade (see for example \cite{felli2005fourth,hebey2001coercivity,MR2819586,qing2006compactness,wei2013non} and the references therein). However, less is known for solutions of \eqref{eqin} in the case where $P_g$ is a Paneitz-Branson type operator with general coefficients. Esposito and Robert \cite{MR1942129} proved the existence of a non trivial solution to \eqref{eqin} under the hypothesis that $n\geq 8$ and $\displaystyle\min_{M} Tr_g (A_g-A_{paneitz})<0$.  In \cite{MR2136976}, Sandeep studied the stability of equation \eqref{eqin} in the following sense : he considered sequences of positive solutions $(u_\alpha)_\alpha$ of
$$\Delta^2_g u_\alpha -div_g (A_\alpha du_\alpha)+a_\alpha u_\alpha = u^{2^\ast -1}_\alpha,\ u_\alpha \in C^{4,\theta},$$
where $A_\alpha$ are smooth $(2,0)$ symmetric tensors and $a_\alpha$ are smooth functions. Sandeep proved that if $A_\alpha$ converges in $C^1(M)$ to a smooth symetric tensor $A_g$, $a_\alpha$ converges in $C^0(M)$ to a smooth positive function $a$ and $u_\alpha$ converges weakly in $H^2(M)$ to a function $u_0$, then $u_0$ is nontrivial provided that $A_g-A_{paneitz}$ is either positive or negative definite (generalizing a result of \cite{hebey2006compactness}). Recently, Pistoia and Vaira \cite{pistoia2012stability} studied the stability of \eqref{eqin} when it is the Paneitz-Branson equation, namely they considered the following equation
\begin{equation}
\label{pisvai}
\Delta^2_g u -div_g ((A_{paneitz}+\varepsilon B)du) +Q_g u =|u|^{2^\ast -2}u,
\end{equation}
where $\varepsilon$ is a small positive parameter and $B$ is a smooth symmetric $(2,0)$ tensor. 
They proved that if $(M,g)$ is not conformally flat, $n\geq 9$ and there exists $\xi_0 \in M$ a $C^1$ stable critical point  (see below for the definition) of the function $\xi \rightarrow \dfrac{Tr_g B(\xi)}{|Weyl_g (\xi)|_g}$, such that $Tr_g B(\xi_0)>0$, then equation \eqref{pisvai} is not stable, i.e. there exists $\varepsilon_0>0$ such that, for any $\varepsilon\in (0,\varepsilon_0)$, equation \eqref{pisvai} admits a solution $u_\varepsilon$ such that $u_\varepsilon (\xi_0) \underset{\varepsilon \rightarrow 0^+}{\longrightarrow}+\infty$. 
 
The aim of this paper is to investigate the stability in the sense of Deng-Pistoia of \eqref{eqin}. We say that \eqref{eqin} is stable if, for any sequences of real positive numbers $(\varepsilon_\alpha)_\alpha$ such that $\varepsilon_\alpha \underset{\alpha \rightarrow \infty}{\longrightarrow } 0$ and for any sequences of solutions $(u_\alpha)_\alpha \in C^{4,\theta}(M)$, $\theta \in (0,1)$, of
\begin{equation}
\label{eq}
\Delta_g^2 u_\alpha-div_g (A_g du_\alpha) +h u_\alpha=|u_\alpha|^{2^\ast-2-\varepsilon_\alpha}u_\alpha,
\end{equation}
bounded in $H^{2}(M)$, then up to a subsequence, $u_\alpha$ converges in $C^4(M)$ to some smooth function $u$ solution of \eqref{eqin}. Deng and Pistoia \cite{MR2859126} proved that \eqref{eqin} is not stable if
\begin{enumerate}
\item $n\geq 7$, $Tr_g (A_g-A_{paneitz})$ is not constant and $\displaystyle\min_M Tr_g (A_g-A_{paneitz})>0$,
\item or $n\geq 8$ and $\xi_0\in M$ a $C^1$ stable critical point of $Tr_g (A_g-A_{paneitz})$ such that $Tr_g(A_g-A_{paneitz})(\xi_0)>0$. 
\end{enumerate}
Our main result shows that under suitable assumptions, equation \eqref{eqin} is not stable for any $n\geq 5$. 
In fact, inspired by the recent result of Robert and V\'etois \cite{robert2012sign} on scalar curvature type equations, we investigate the existence of
families $(u_\varepsilon)_\varepsilon \in C^{4,\theta}(M)$ of  blow-up solutions to \eqref{eq} of type $(u_0-\bb_\varepsilon)$.
Following the terminology of Robert and V\'etois, 
we say that a blow-up sequence $(u_\varepsilon)_\varepsilon$ is of type $(u_0-\bb_\varepsilon )$ if there exists $u_0\in C^{4,\theta}(M)$ 
and a bubble $\bb_\varepsilon(x)=[n(n-4)(n^2-4)]^{\frac{n-4}{8}} \left(\dfrac{\mu_\varepsilon}{\mu_\varepsilon+d_g(x,x_\varepsilon)^2}\right)^{\frac{n-4}{2}}$,
where $x,x_\varepsilon\in M$ and $\mu_\varepsilon\in \R^+$ is such that $\mu_\varepsilon \underset{\varepsilon\rightarrow 0}{\longrightarrow }0$, such that
$$u_\varepsilon=u_0-\bb_\varepsilon+o(1),$$
where $o(1)\underset{\varepsilon\rightarrow 0}{\longrightarrow }0$.
Before stating more precisely the results, we would like to recall that a solution of \eqref{eq} is called nondegenerate if the kernel of the linearization of the equation is trivial (see \eqref{nondegdef}). Let $\phi\in C^1(M)$, we also recall that a critical point $\xi_0$ of $\phi$ is said $C^1$ 
stable if there exists an open neighborhood $\Omega$ of $\xi_0$ such that, for any point $\xi \in \bar{\Omega}$, there holds $\nabla_g \phi (\xi)=0$ if and only if $\xi=\xi_0$ and such that the Brower degree $deg (\nabla_g \phi, \Omega ,0)\neq 0$. We obtain :

\begin{thm}
\label{thm1}
Let $(M,g)$ be a compact riemannian manifold of dimension $n$, $A_g$ and $h$ be such that $P_g$ is coercive. Let $u_0\in C^{4,\theta}$, $\theta \in (0,1)$, be a positive nondegenerate solution of \eqref{eqin}.
Assume in addition that one of the following condition holds:
\begin{enumerate}
\item $5\leq n<7$, \item  $8\leq n \leq 13$ and there exists $\xi_0\in M$ a $C^1$ stable critical point of 
\begin{equation}
\begin{multlined}
\varphi (\xi) =\dfrac{(n-1)}{(n-6)(n^2-4)}(Tr_g (A_g- A_{paneitz}))(\xi) \ \ \ \ \ \ \ \ \ \ \ \ \ \  \\ 
+ \dfrac{2^{n} u_0(\xi)\omega_{n-1}}{(n+2) (n (n-4)(n^2-4))^{\frac{n-4}{8}} \omega_n}1_{n=8},\ \xi\in M,
\end{multlined}
\end{equation}
such that $\varphi (\xi_0)>0$,
\item  $n>13$ and $\displaystyle\min_{M}Tr_g(A_g-A_{paneitz})>0$,
\end{enumerate}
then, for any $\varepsilon>0$, there exists a solution $u_\varepsilon$ of type $u_0-\bb_\varepsilon$ to \eqref{eq}. In particular, \eqref{eq} is not stable.
\end{thm}
Let us notice that in the geometric case i.e. when $A_g=A_{paneitz}$, the previous theorem only applies if $5\leq n\leq 8$. However, with a small modification of the proof, we can construct a solution of type $u_0-\bb_\varepsilon$ to \eqref{eq} when $5\leq n\leq 11$ and $A_g=A_{paneitz}$. More precisely, we prove the following result :
\begin{thm}
\label{thm2}
Let $(M,g)$ be a compact riemannian manifold of dimension $n$, $A_g$ and $h$ be such that $P_g$ is coercive. Let $u_0\in C^{4,\theta}$, $\theta \in (0,1)$, be a positive nondegenerate solution of \eqref{eqin}. Assume that $A_g=A_{paneitz}$. Then, for any $5\leq n\leq 11$ and any $\varepsilon>0$, there exists a solution $u_\varepsilon$ of type $u_0-\bb_\varepsilon$ to \eqref{eq}. In particular, \eqref{eq} is not stable.
\end{thm}

The proof of the theorems relies on a well known Lyapunov-Schmidt reduction procedure which permits to reduce the problem to a finite dimensional one for which we defined a reduced energy.
The solutions to \eqref{eq} will then be obtained as critical points of this reduced energy. We refer to \cite{ambrosetti2006perturbation} 
 and the references therein for more information on the Lyapunov-Schmidt reduction procedure. 
We would like to emphasize that the proof of Theorem \ref{thm1} is inspired by the previous work of Robert and V\'etois \cite{robert2012sign}. 
Thus we will keep their notations. We also want to point out that we use without proof computations done in Deng and Pistoia \cite{MR2859126} (for more details on these computations, see their paper).  
The plan of the paper is the following : in section  2 we introduce notations and perform the finite dimensional reduction. In section 3 we study the reduced problem and prove Theorem \ref{thm1}.
The error estimate and the $C^1$ uniform asymptotic expansion of the reduced energy are done in the appendix.
\\
\textbf{Acknowledgements} : The authors would like to thank F. Robert for his comments and suggestions on a preliminary version of this paper.
\section{Finite dimensional reduction.}
Let $(\xi_\alpha)_\alpha$ be a sequence of points of $M$. In all the following, we will suppose up to extracting a subsequence that, 
for $\alpha$ large enough, all the points $\xi_\alpha$ belong to a small open set $\Omega$ of $M$ in which there exists a smooth orthogonal frame. 
Thus, we will identify the tangent spaces $T_\xi M$ with $\R^n$ for all $\xi\in \Omega$. 
We recall that we suppose that $P_g$ is coercive. 

In all the following, we will denote by $\left\langle .,.\right\rangle_{P_g}$, the scalar product, for $u,v\in H^2(M)$,
$$\left\langle u,v\right\rangle_{P_g}= \int_M \Delta_g u \Delta_g v dV +\int_M A_g (\nabla_g u ,\nabla_g v)dV+\int_M h uv dV,$$
where here and in the following $dV$ stands for the volume element with respect to the metric $g$, and  $\left\|.\right\|_{P_g}$, for the associated norm which is then equivalent to the standard norm of $H^2(M)$.
We denote by $i^\ast : L^{\frac{2n}{n+4}}(M)\rightarrow H^2(M)$ the adjoint operator of the embedding $i: H^2(M)\rightarrow L^{\frac{2n}{n-4}}(M)$, 
i.e. for all $\varphi \in L^{\frac{2n}{n+4}}(M)$, the function $u=i^\ast (\varphi)\in H^2(M)$ is the unique solution 
of $\Delta^2_g u-div_g (A_g du) +h u =\varphi$. Using this notation, we see that equation \eqref{eq} can be rewritten as, for $u\in H^2(M)$,
$$u=i^\ast(f_\varepsilon (u)),$$
where $f_\varepsilon (u)=|u|^{2^\ast -2-\varepsilon}u$. Before proceeding we recall some basic facts. It is well known (see \cite{lin1998classification}) 
that all solutions $u\in H^2(\R^n)$ 
of the equation
$$\Delta^2_{eucl} u=u^{2^\ast-1}=u^{\frac{n+4}{n-4}}\ in\ \R^n$$
are given by
$$U_{\delta, y}(x)=\delta^{\frac{4-n}{2}}U(\frac{x-y}{\delta}),\ \delta >0,\ y\in \R^n$$
where 
\begin{equation}
\label{defalpha}
U(x)= [n(n-4)(n^2-4)]^{\frac{n-4}{8}} \left(\dfrac{1}{1+|x|^2}\right)^{\frac{n-4}{2}}=\alpha_n \left(\dfrac{1}{1+|x|^2}\right)^{\frac{n-4}{2}}.
\end{equation}
It is also well known (see \cite{MR1694339}) that all solutions $v\in H^2(\R^n)$  
of
$$\Delta^{2}_{eucl} v= (2^\ast -1)U^{2^\ast-2}v$$
are linear combinations of
$$V_0(x)= \alpha_n \dfrac{n-4}{2}\dfrac{|x|^2-1}{(1+|x|^2)^{\frac{n-2}{2}}}$$
and
$$V_i(x)= \alpha_n (n-4)\dfrac{x_i}{(1+|x|^2)^{\frac{n-2}{2}}},\ i=1,\ldots ,n. $$
Let $\chi : \R\rightarrow \R$ be a smooth cutoff function such that $0\leq \chi\leq 1$, $\chi(x)=1$ if $x\in [-\dfrac{r_0}{2},\dfrac{r_0}{2}]$ and $\chi(x)=0$ if $x\in\R \backslash (-r_0,r_0)$. We define, for any real $\delta$ strictly positive, $\xi \in M$ and $x\in M$,
$$W_{\delta,\xi}(x)=\chi (d_g(x,\xi))\delta^{\frac{4-n}{2}}U(\delta^{-1}\exp_\xi^{-1}(x)),$$
where $d_g(x,\xi)$ stands for the distance from $x$ to $\xi$ with respect to the metric $g$ and $\exp_\xi$ is the exponential map with respect to the metric $g$. We also define, for any real $\delta$ strictly positive, $\xi \in M$ and $x\in M$,
$$Z_{\delta,\xi}(x)=\chi (d_g(x,\xi))\delta^{\frac{n-4}{2}}\dfrac{d(x,\xi)^2 -\delta^2}{(\delta^2+d(x,\xi)^2)^{\frac{n-2}{2}}},$$
and, for $\omega \in T_\xi M$,
$$Z_{\delta,\xi,\omega}(x)=\chi (d_g(x,\xi))\delta^{\frac{n-2}{2}}\dfrac{\left\langle \exp_\xi^{-1}x,\omega\right\rangle_g}{(\delta^2+d(x,\xi)^2)^{\frac{n-2}{2}}}. $$
We denote by $\Pi_{\delta,\xi}$ respectively $\Pi_{\delta,\xi}^{\bot}$ the projection of $H^2(M)$ onto
$$K_{\delta,\xi}=\mathrm{span}\left\{Z_{\delta,\xi},(Z_{\delta,\xi,e_i})_{i=1..n}\right\}$$
respectively
\begin{equation}\label{nodegen}K_{\delta,\xi}^{\bot}=\left\{\phi \in H^2(M)\slash \left\langle \phi , Z_{\delta,\xi}\right\rangle_{P_g} =0\ and\ \left\langle \phi , Z_{\delta,\xi,\omega}\right\rangle_{P_g} =0,\ \forall \omega\in T_\xi M  \right\}.\end{equation}
We recall that a solution $u_0$ of \eqref{eq} is nondegenerate if the linearization of the equation has trivial kernel, that is
\begin{equation}
\label{nondegdef}
K=\left\{\varphi \in C^{4,\theta}(M) \slash P_g \varphi= (2^\ast -1)|u_0|^{2^\ast -2}\varphi  \right\}=\left\{0\right\}.
\end{equation}
We are looking for solution $u$ to \eqref{eq} of the form $$u=u_0-W_{\delta_\varepsilon(t_\varepsilon),\xi_\varepsilon}+\phi_{\delta_\varepsilon(t_\varepsilon),\xi_\varepsilon},$$
where $u_0$ is a nondegenerate positive solution of \eqref{eq}, $\phi_{\delta_\varepsilon(t_\varepsilon),\xi_\varepsilon}\in K_{\delta_\varepsilon(t_\varepsilon),\xi_\varepsilon}^{\bot}$ and 
\begin{equation}\label{delta}\delta_\varepsilon(t_\varepsilon)=\left\{\begin{array}{ll}\sqrt{t_\varepsilon\varepsilon}\ \mathrm{if}\ n\geq 8\\
(t_\varepsilon\varepsilon)^{\frac{2}{n-4}}\ \mathrm{if}\ 5\leq n\leq 8 \end{array}\right.,\ t_\varepsilon>0.\end{equation} It is easy to see that equation \eqref{eq} is equivalent to the following system
\begin{equation}
\label{eq1}
\Pi_{\delta_\varepsilon(t),\xi }(u_0 - W_{\delta_\varepsilon(t),\xi }+\phi_{\delta_\varepsilon(t),\xi } -i^\ast (f_\varepsilon(u_0-W_{\delta_\varepsilon(t),\xi }+\phi_{\delta_\varepsilon(t),\xi })))=0,
\end{equation}
and
\begin{equation}
\label{eq2}
\Pi_{\delta_\varepsilon (t),\xi }^\bot(u_0 - W_{\delta_\varepsilon(t),\xi }+\phi_{\delta_\varepsilon(t),\xi } -i^\ast (f_\varepsilon(u_0-W_{\delta_\varepsilon(t),\xi }+\phi_{\delta_\varepsilon(t),\xi })))=0.
\end{equation}
We begin by solving \eqref{eq2}.
\begin{prop}
\label{prop4.1}
Let $u_0\in C^{4,\theta}(M)$ be a nondegenerate positive solution of \eqref{eq}. Given two real numbers $a<b$, there exists a positive constant $C_{a,b}$ such that for $\varepsilon$ small, for any $t\in [a,b]$
and any $\xi \in M$, there exists a unique function $\phi_{\delta_\varepsilon (t),\xi} \in K_{\delta_{\varepsilon}(t), \xi}^\bot$
which solves equation \eqref{eq2} and satisfies
\begin{eqnarray}
\label{eqprop4.1}
\left\|\phi_{\delta_\varepsilon (t),\xi}\right\|_{P_g} \leq C_{a,b} \varepsilon |\ln\varepsilon |.
\end{eqnarray}
Moreover, $\phi_{\delta_\varepsilon (t),\xi}$ is continuously differentiable with respect to $t$ and $\xi$.
\end{prop}
In order to prove the previous proposition, we set, for $\varepsilon$ small, for any positive real number $\delta$ and $\xi \in M$, the map $L_{\varepsilon , \delta,\xi}: K_{\delta, \varepsilon}^\bot \rightarrow K_{\delta, \varepsilon}^\bot$ 
defined by, for $\phi \in K_{\delta, \varepsilon}^\bot$,
$$L_{\varepsilon , \delta,\xi}(\phi)= \Pi_{\delta,\xi }^\bot(\phi - i^\ast (f^\prime_{\varepsilon}(u_0 - W_{\delta,\xi})\phi)).$$
We will first prove that this map is inversible for $\delta$ and $\varepsilon$ small.
\begin{lem}
\label{inver}
There exists a positive constant $C_{a,b}$ such that for $\varepsilon$ small, for any $t\in [a,b]$, any $\xi \in M$ and any $\phi \in K_{\delta, \varepsilon}^\bot$, we have
$$\left\|L_{\varepsilon_\alpha , \delta_{\varepsilon_\alpha}(t_\alpha),\xi_\alpha}(\phi)\right\|_{P_g} \geq C_{a,b}\left\|\phi\right\|_{P_g}.$$
\end{lem}
\begin{proof}
Assume by contradiction that there exist two sequences of positive real numbers $(\varepsilon_\alpha)_\alpha$ and $(t_\alpha)_\alpha$ such that $\varepsilon_\alpha\underset{\alpha\rightarrow +\infty}{\longrightarrow} 0$
and $a\leq t_\alpha \leq b$, a sequence of points $(\xi_\alpha))_\alpha$ of $M$ and a sequence of functions $(\phi_\alpha)_\alpha$ such that
\begin{equation}
\label{rob25}
\phi_\alpha \in  K_{\delta_{\varepsilon_\alpha}(t_\alpha), \xi_\alpha}^\bot,\ \left\|\phi_\alpha\right\|_{P_g}=1\ \mbox{ and}\ \left\|L_{\varepsilon_\alpha , \delta_{\varepsilon_\alpha}(t_\alpha),\xi_\alpha}(\phi_\alpha)\right\|_{P_g}\underset{\alpha \rightarrow \infty}{\longrightarrow} 0.
\end{equation}
To simplify notations, we set $L_\alpha = L_{\varepsilon_\alpha , \delta_{\varepsilon_\alpha}(t_\alpha),\xi_\alpha}$, $W_\alpha=W_{\delta_{\varepsilon_\alpha}(t_\alpha),\xi_\alpha}$, $Z_{0,\alpha}=Z_{\delta_{\varepsilon_\alpha}(t_\alpha),\xi_\alpha}$ and $Z_{i,\alpha}= Z_{\delta_{\varepsilon_\alpha}(t_\alpha),\xi_\alpha,e_i}$ for $i=1,\ldots ,n$
where $e_i$ is the $i$-th vector in the canonical basis of $\R^n$. 
By definition of $L_\alpha$, there exist real numbers $\lambda_{i,\alpha}$, $i=0,\ldots ,n$ such that
\begin{equation}
\label{rob28}
\phi_\alpha - i^\ast (f^\prime_{\varepsilon_\alpha}(u_0 - W_\alpha)\phi_\alpha )-L_\alpha(\phi_\alpha)=\sum_{i=0}^n \lambda_{i,\alpha} Z_{i,\alpha}. 
\end{equation}
Standard computations give
\begin{equation}
\label{ortho}
\left\langle Z_{i,\alpha}, Z_{j,\alpha}\right\rangle_{P_g} \underset{\alpha \rightarrow \infty}{\longrightarrow} \left\|\Delta_{eucl} V_i\right\|_{L^{2}(\R^n)}^2 \delta_{ij},
\end{equation}
where $\delta_{ij}$ stands for the Kronecker symbol. Therefore, taking the scalar product of \eqref{rob28} with $Z_{i,\alpha}$, using the previous limit and recalling that $\phi_\alpha$ and $L_\alpha(\phi_\alpha)$ belong to $K_{\delta_{\varepsilon_\alpha}(t_\alpha), \xi_\alpha}^\bot$, we deduce that 
\begin{equation}
\label{rob30}
\int_M f^\prime_{\varepsilon_\alpha}(u_0 - W_\alpha)\phi_\alpha Z_{i,\alpha}dV=-\lambda_{i,\alpha}\left\|\Delta_{eucl} V_i\right\|_{L^{2}(\R^n)}^2+\left(\sum_{i=0}^n |\lambda_{i,\alpha}|\right) o(1), 
\end{equation}
where, here and in the following, $o(1)\underset{\alpha \rightarrow +\infty}{\longrightarrow } 0$. It is easy to see using the definition of $W_\alpha$ and $Z_{i,\alpha}$ and a change of variables that, for $\alpha$ large enough,
\begin{eqnarray}
\label{rob35}
&&\int_M f^\prime_{\varepsilon_\alpha}(u_0 - W_\alpha)\phi_\alpha Z_{i,\alpha}dV\nonumber \\
&=&\int_M f^\prime_{\varepsilon_\alpha}( W_\alpha)\phi_\alpha Z_{i,\alpha}dV+o(1)\\
&=& (2^\ast -1-\varepsilon_\alpha)\delta_{\varepsilon_{\alpha}}(t_\alpha)^{\varepsilon_\alpha\frac{n+4}{2}}\int_{\R^n} \chi_\alpha^{2^\ast -2-\varepsilon_\alpha} U^{2^\ast -2-\varepsilon_\alpha}V_i \tilde{\phi}_\alpha dV_{\tilde{g}_\alpha}+o(1),\nonumber
\end{eqnarray}
where $\chi_\alpha= \chi (\delta_{\varepsilon_\alpha}(t_\alpha)|x|)$, $\tilde{\phi}_\alpha (x) =\delta_{\varepsilon_\alpha}(t_\alpha)^{\frac{n-4}{2}}\chi_\alpha \phi_\alpha (\exp_{\xi_\alpha}(\delta_{\varepsilon_\alpha}(t_\alpha)x )) $ 
and $\tilde{g}_\alpha (x)= \exp^\ast_{\xi_\alpha} g (\delta_{\varepsilon_\alpha}(t_\alpha)x)$.
Since $(\phi_\alpha)_\alpha$ is bounded in $H^2(M)$, passing to a subsequence if necessary,
we can assume that $(\tilde{\phi}_\alpha)_\alpha$ converges weakly to a function $\tilde{\phi}\in H^2(\R^n)$.
Letting $\alpha\rightarrow +\infty$ in \eqref{rob35}, we deduce that
\begin{equation}
\label{rob36}
\int_M f^\prime_{\varepsilon_\alpha}(u_0 - W_\alpha)\phi_\alpha Z_{i,\alpha}dV\underset{\alpha \rightarrow \infty}{\longrightarrow}(2^\ast -1) \int_{\R^n} U^{2^\ast -2}V_i \tilde{\phi}dV_{g_{eucl}}=0,
\end{equation}
where we used that $V_i$ is solution of $\Delta^{2}_{eucl} V_i= \dfrac{n+4}{n-4}U^{2^\ast -2}V_i$ in $\R^n$ and $\phi_\alpha \in  K_{\delta_{\varepsilon_\alpha}(t_\alpha), \xi_\alpha}^\bot$ to obtain the last equality. Therefore, from \eqref{rob30} and \eqref{rob36}, we have
$$\lambda_{i,\alpha}=o(1)+o(\sum_{i=0}^n |\lambda_{i,\alpha}|).$$
From \eqref{rob28}, this implies
$$\phi_\alpha - i^\ast (f^\prime_{\varepsilon_\alpha}(u_0 - W_\alpha)\phi_\alpha )-L_\alpha(\phi_\alpha)\underset{\alpha \rightarrow \infty}{\longrightarrow}0.$$
Since by assumption $\left\|L_{\varepsilon_\alpha , \delta_{\varepsilon_\alpha}(t_\alpha),\xi_\alpha}(\phi_\alpha)\right\|_{P_g}\underset{\alpha \rightarrow \infty}{\longrightarrow} 0$, we finally obtain that
\begin{equation}
\label{rob26}
\left\|\phi_\alpha  - i^\ast (f^\prime_{\varepsilon_\alpha}(u_0 - W_\alpha)\phi_\alpha )\right\|_{P_g}\underset{\alpha \rightarrow \infty}{\longrightarrow}0. 
\end{equation}
Since $(\phi_\alpha)_\alpha$ is bounded in $H^2(M)$, up to taking a subsequence, we can assume that $\phi_\alpha$ converges weakly in $H^2(M)$ to a function $\phi \in H^2(M)$. Then, using \eqref{rob26}, we get, for any $\varphi\in H^2(M)$,
\begin{align}
\label{rob37}
\left|\left\langle \varphi ,\phi_\alpha\right\rangle_{P_g} -\int_M f_{\varepsilon_\alpha}^\prime (u_0-W_\alpha) \varphi \phi_\alpha dV\right|&
=
{}
\left|\left\langle \varphi ,\phi_\alpha-i^\ast(f_{\varepsilon_{\alpha}}^\prime(u_0-W_\alpha)\phi_\alpha)\right\rangle_{P_g} \right|\nonumber\\
&\leq {}
\left\|\varphi \right\|_{P_g} \left\|\phi_\alpha-i^\ast(f_{\varepsilon_{\alpha}}^\prime(u_0-W_\alpha)\phi_\alpha)\right\|_{P_g}\nonumber\\
&=
o( \left\|\varphi \right\|_{P_g}).
\end{align}
We deduce from this that $\phi$ is a weak solution of $P_g \phi= (2^\ast -1) u_0^{2^\ast-2}\phi$. Since $u_0$ is a nondegenerate solution of \eqref{eq}, we obtain that $\phi=0$. 
Therefore, $\phi_\alpha\underset{\alpha \rightarrow \infty}{\rightharpoonup}0$ weakly in $H^2(M)$. 
Now we will show that $\tilde{\phi}_\alpha \underset{\alpha \rightarrow \infty}{\rightharpoonup}0$ weakly in $H^{2}(\R^n)$.
Let $\tilde{\varphi}$ be a smooth function with compact support in $\R^n$, we will use \eqref{rob37} with, for $x\in M$, 
$$\varphi (x)=\chi (d_{g_{\xi_\alpha}}(x,\xi_\alpha))\delta_{\varepsilon_\alpha}(t_\alpha)^{\frac{4-n}{2}}\tilde{\varphi} (\delta_{\varepsilon_\alpha}(t_\alpha)^{-1} \exp_{\xi_\alpha}^{-1}(x)).$$
Thus, applying \eqref{rob37} to the previous $\varphi$ and using a change of variable, we have,
\begin{eqnarray}
\label{rob38}
&&\int_{\R^n} \Delta_{\tilde{g}_\alpha} \tilde{\phi}_\alpha \Delta_{\tilde{g}_\alpha} \tilde{\varphi} dV_{\tilde{g}_\alpha}+\delta_{\varepsilon_\alpha}(t_\alpha)^{2}\int_{\R^n} A_{\tilde{g}_\alpha} (\nabla_{\tilde{g}_\alpha}\tilde{\phi}_\alpha,\nabla_{\tilde{g}_\alpha} \tilde{\varphi})dV_{\tilde{g}_\alpha}\nonumber\\
&+&\delta_{\varepsilon_\alpha}(t_\alpha)^{4}\int_{\R^n}h(\exp_{\xi_\alpha}(\delta_{\varepsilon_\alpha}(t_\alpha)x))\tilde{\phi}_\alpha \tilde{\varphi} dV_{\tilde{g}_\alpha}\\
&=& \delta_{\varepsilon_\alpha}(t_\alpha)^{4}\int_{\R^n} f_{\varepsilon_\alpha}^\prime (u_{0,\alpha} -W_\alpha(\exp_{\xi_\alpha} (\delta_{\varepsilon_\alpha}(t_\alpha)x)))\tilde{\phi}_\alpha \varphi dV_{\tilde{g}_\alpha}+o(1),\nonumber
\end{eqnarray}
where $u_{0,\alpha}(.)=u_0(\exp_{\xi_\alpha}(\delta_{\varepsilon_\alpha}(t_\alpha).))$. Now it is easy to see that, letting $\alpha\rightarrow \infty$ in \eqref{rob38}, 
 $$\int_{\R^n} \Delta_{eucl} \tilde{\phi} \Delta_{eucl} \tilde{\varphi} dV_{g_{eucl}} = (2^\ast-1) \int_{\R^n}U^{2^\ast -2}\tilde{\phi}\tilde{\varphi} dV_{g_{eucl}}.$$
 Thus $\tilde{\phi}$ is a weak solution of $\Delta^{2}_{eucl} \tilde{\phi}= \dfrac{n+4}{n-4}U^{2^\ast -2}\tilde{\phi}$. So, from \cite{MR1694339}, 
 we know that there exists $\lambda_i\in \R$, $i=0,\ldots ,n$, such that $\tilde{\phi}=\sum_{i=0}^n \lambda_i V_i$.
 Since $\phi_\alpha \in K_{\delta_{\varepsilon_\alpha}(t_\alpha), \xi_\alpha}^\bot$, using the same argument as in \eqref{rob36}, 
 we deduce that $\tilde{\phi}\equiv 0$. 
 Using one more time \eqref{rob37} with $\varphi=\phi_\alpha$,
 a change of variables and since $\phi_\alpha\underset{\alpha \rightarrow \infty}{\rightharpoonup}0$
 weakly in $H^2(M)$ and $\tilde{\phi}_\alpha \underset{\alpha \rightarrow \infty}{\rightharpoonup}0$ 
 weakly in $H^{2}(\R^n)$, we get
\begin{eqnarray*}
\left\|\phi_\alpha\right\|_{P_g}^2 &=&(2^\ast -1-\varepsilon_\alpha) \int_M |u_0 - W_\alpha|^{2^\ast -2-\varepsilon_\alpha} \phi_\alpha^2 dV+o(1)\\
&\leq & C \int_M \phi_\alpha^2 dV +C \int_M |W_\alpha|^{2^\ast -2-\varepsilon_\alpha}\phi_\alpha^2dV+o(1)\\
&\leq & C \int_M \phi_\alpha^2 dV +C \int_M |U|^{2^\ast -2-\varepsilon_\alpha}\tilde{\phi}_\alpha^2dV_{\tilde{g}_\alpha}+o(1)\underset{\alpha\rightarrow \infty}{\longrightarrow}0.
\end{eqnarray*}
This yields to a contradiction with \eqref{rob25}.

\end{proof}

\begin{proof}[Proof of Proposition \ref{prop4.1}.]
It is easy to see that equation \eqref{eq2} is equivalent to
$$L_{\varepsilon , \delta_\varepsilon (t),\xi}(\phi)=N_{\varepsilon , \delta_\varepsilon (t),\xi}(\phi)+R_{\varepsilon , \delta_\varepsilon (t),\xi},$$
where
\begin{equation*}
 \begin{multlined}N_{\varepsilon , \delta_\varepsilon (t),\xi}(\phi)=\Pi_{\delta_\varepsilon(t),\xi}^{\bot}(i^\ast (f_\varepsilon (u_0 -W_{\delta_\varepsilon (t),\xi}+\phi))
 -f_\varepsilon(u_0-W_{\delta_\varepsilon (t),\xi})\\-f_\varepsilon^\prime (u_0-W_{\delta_\varepsilon (t),\xi})\phi),
 \end{multlined}\end{equation*}

and
$$R_{\varepsilon, \delta_\varepsilon (t),\xi}= \Pi_{\delta_\varepsilon(t),\xi}^{\bot}(i_\varepsilon^\ast (f_\varepsilon (u_0 -W_{\delta_\varepsilon (t),\xi}))-u_0+W_{\delta_\varepsilon (t),\xi}).$$
Let $T_{\varepsilon , \delta_{\varepsilon}(t),\xi}: K_{\delta_{\varepsilon_\alpha}(t_\alpha), \xi_\alpha}^\bot\rightarrow K_{\delta_{\varepsilon_\alpha}(t_\alpha), \xi_\alpha}^\bot$ be the application defined by
$$T_{\varepsilon , \delta_{\varepsilon}(t),\xi}(\phi)=L_{\varepsilon , \delta_\varepsilon (t),\xi}^{-1}(N_{\varepsilon , \delta_\varepsilon (t),\xi}(\phi)+R_{\varepsilon , \delta_\varepsilon (t),\xi} ),$$
and
$$B_{\varepsilon , \delta_\varepsilon (t),\xi}(\gamma)=\left\{\phi \in K_{\delta_{\varepsilon_\alpha}(t_\alpha), \xi_\alpha}^\bot |  \left\|\phi\right\|_{P_g}\leq \gamma\left\|R_{\varepsilon , \delta_\varepsilon (t),\xi} \right\|_{P_g} \right\},$$
where $\gamma$ is a positive constant which will be chosen later in order to apply the fixed point theorem for $T_{\varepsilon , \delta_{\varepsilon}(t),\xi}$ restricted to $B_{\varepsilon , \delta_\varepsilon (t),\xi}(\gamma)$. Since, from Lemma \ref{inver}, the map $L_{\varepsilon , \delta_\varepsilon (t),\xi}$ is inversible and has a continuous inverse, we have
\begin{equation}
\label{dp38}
\left\|T_{\varepsilon , \delta_\varepsilon (t),\xi}(\phi)\right\|_{P_g}\leq C ( \left\|N_{\varepsilon , \delta_\varepsilon (t),\xi}(\phi)\right\|_{P_g}+\left\|R_{\varepsilon , \delta_\varepsilon (t),\xi}\right\|_{P_g}), 
\end{equation}
and
\begin{equation}
\label{dp39}
 \left\|T_{\varepsilon , \delta_{\varepsilon}(t),\xi}(\phi_1)-T_{\varepsilon , \delta_{\varepsilon}(t),\xi}(\phi_2)\right\|_{P_g}\leq C \left\|N_{\varepsilon , \delta_\varepsilon (t),\xi}(\phi_1)-N_{\varepsilon , \delta_\varepsilon (t),\xi}(\phi_2) \right\|_{P_g}.
 \end{equation}
Since $i^\ast: L^{\frac{2n}{n+4}}(M)\rightarrow H^2(M)$ is continuous, we get
\begin{equation*}\begin{multlined}\left\|N_{\varepsilon , \delta_\varepsilon (t),\xi}(\phi) \right\|_{P_g}\leq 
\\
C\left\|f_\varepsilon
(u_0 -W_{\delta_\varepsilon (t),\xi}+\phi))-f_\varepsilon(u_0-W_{\delta_\varepsilon (t),\xi})-f_\varepsilon^\prime 
(u_0-W_{\delta_\varepsilon (t),\xi})\phi \right\|_{L^{\frac{2n}{n+4}}},\end{multlined}\end{equation*}
where, here and in the following, $\left\|.\right\|_{L^p}=\left\|.\right\|_{L^p(M)}$, $p\in \R^+$. Using the mean value theorem, H\"older and Sobolev inequalities, we have, for $\tau \in (0,1)$,
\begin{align}
\left\|N_{\varepsilon , \delta_\varepsilon (t),\xi}(\phi) \right\|_{P_g} & \leq 
  C \left\|\left[f_\varepsilon^\prime (u_0-W_{\delta_\varepsilon (t),\xi} +\tau \phi) -f_\varepsilon^\prime (u_0-W_{\delta_\varepsilon (t),\xi})\right] (\phi)\right\|_{L^{\frac{2n}{n+4}} \!}\nonumber\\
\leq
& C \left\|f_\varepsilon^\prime (u_0-W_{\delta_\varepsilon (t),\xi} +\tau \phi) -f_\varepsilon^\prime (u_0-W_{\delta_\varepsilon (t),\xi})\right\|_{L^{\frac{n}{4}}} \left\|\phi\right\|_{L^{2^\ast}}\!.\nonumber
\end{align}
We will use here and through the paper the following easy consequences of Taylor's expansion \cite[lemma 2.2]{Li},
for all $\alpha >0$, $\beta\in \R$,
\begin{equation}
\label{dp1}
||\alpha+\beta|^\theta - \alpha^\theta|\leq  \left\{\begin{array}{ll}C_\theta \min \left\{|\beta|^\theta , \alpha^{\theta-1}|\beta|\right\}\ \mathrm{if}\ 0<\theta \leq 1,\\
C_\theta (\alpha^{\theta-1}|\beta|+|\beta|^\theta)\ \mathrm{if}\ \theta >1,
\end{array}
\right.
\end{equation}
and 
\begin{equation}\label{dp2}
 ||\alpha+\beta|^\theta(\alpha+\beta) - \alpha^{\theta+1}-(1+\theta)\alpha^\theta\beta|\leq  \left\{
 \begin{array}{ll}C_\theta\mathrm{min}\left\{|\beta|^{\theta+1} , \alpha^{\theta-1}|\beta|^2\right\}\ \mathrm{if}\ \theta< 1,\\
C_\theta \mathrm{max}\{|\beta|^{\theta+1}, \alpha^{\theta-1}|\beta|^2\}\ \mathrm{if}\ \theta \geq 1.
\end{array}
\right.
\end{equation}
Thus, we obtain
\begin{equation}
\label{dp42}
\left\|N_{\varepsilon , \delta_\varepsilon (t),\xi}(\phi)\right\|_{P_g}\leq  \left\{\begin{array}{ll} C \left\|\phi\right\|_{P_g}^{2^\ast -1-\varepsilon}\ \mathrm{if}\ n\geq 12, \\
C (\left\|u_0-W\right\|^{2^\ast -3-\varepsilon}_{L^{2^\ast}} \left\|\phi\right\|_{P_g}^{2}+ \left\|\phi\right\|_{P_g}^{2^\ast -1-\varepsilon})\ \mathrm{if}\ 5\leq n <12 .\end{array}
\right.
\end{equation}
\noindent From the mean value theorem, H\"older and Sobolev inequalities, and \eqref{dp1}, we also get, for some $\tau \in (0,1)$,
\begin{align}
\label{dp43}
 &\!\!\!\!\!\left\|N_{\varepsilon , \delta_\varepsilon (t),\xi}(\phi_1)-N_{\varepsilon , \delta_\varepsilon (t),\xi}(\phi_2) \right\|_{P_g}\\
&\ \ \quad \ \ \quad\leq C\left\|  f_\varepsilon(u_0-W_{\delta_\varepsilon (t),\xi}+\phi_1)-f_\varepsilon(u_0-W_{\delta_\varepsilon (t),\xi}+\phi_2)\right.\nonumber\\
&\ \ \ \ \quad \ \ \quad \ \ \quad \ \ \quad \ \ \quad \ \ \quad \ \ \quad \ \ \quad\ \ \quad \ -\left.f_\varepsilon^\prime (u_0-W_{\delta_\varepsilon (t),\xi}) (\phi_1-\phi_2)\right\|_{L^{\frac{2n}{n+4}}}\nonumber\\
&\ \ \quad \ \ \quad\leq  C \left\|\left[f_\varepsilon^\prime (u_0-W_{\delta_\varepsilon (t),\xi} +\tau \phi_2+(1-\tau)\phi_1)\right.\right. \nonumber\\
 &\ \ \ \quad \ \ \quad \ \ \quad \ \ \quad \ \ \quad \ \ \quad \ \ \quad \ \ \quad\ \ \quad \               \left.\left.    -f_\varepsilon^\prime (u_0-W_{\delta_\varepsilon (t),\xi})\right] (\phi_1-\phi_2)\right\|_{L^{\frac{2n}{n+4}}}\nonumber\\
&\ \ \quad \ \ \quad\leq  C \left\|f_\varepsilon^\prime (u_0-W_{\delta_\varepsilon (t),\xi} +\tau \phi_2+(1-\tau)\phi_1) -f_\varepsilon^\prime (u_0-W_{\delta_\varepsilon (t),\xi})\right\|_{L^{\frac{n}{4}}}\nonumber\\ 
&\ \ \quad \ \ \ \ \quad \ \ \quad \ \ \quad \ \ \quad \ \ \quad \ \ \quad \ \ \quad \ \ \quad\ \ \quad \ \ \quad \ \ \quad\ \ \quad \  \ \times\left\|\phi_1-\phi_2\right\|_{L^{2^\ast}}\nonumber\\
 &\ \ \quad \ \ \quad\leq  \left\{\begin{array}{l}
 C (\left\|\phi_1\right\|_{P_g}^{2^\ast -2-\varepsilon}+\left\|\phi_2\right\|_{P_g}^{2^\ast -2-\varepsilon})  \left\|\phi_1-\phi_2\right\|_{P_g}\ \ \ \  \ \ \ \ \ \ \mathrm{if}\ n\geq 12,\\
 C(\left\|u_0- W_{\delta_\varepsilon (t),\xi}\right\|_{L^{2^\ast}(M)}+\left\|\phi_1\right\|_{P_g}+\left\|\phi_2\right\|_{P_g} )^{2^\ast -3-\varepsilon}\!\!\!\!\!\!\!\!\!\!\!\!\!\!\!\!\!\!\!\!\!\!\!\! \\
 \ \ \ \ \ \ \ \  \ \ \ \ \ \ \ \times (\left\|\phi_1\right\|_{P_g}+\left\|\phi_2\right\|_{P_g})\left\|\phi_1-\phi_2\right\|_{P_g}\ \ \ \ \mathrm{if}\ 5\leq n< 12 
 \end{array}
\right.
 \nonumber
\end{align}
Since $\left\|u_0- W_{\delta_\varepsilon (t),\xi}\right\|_{L^{2^\ast}}=O(1)$, it follows from \eqref{dp38}, \eqref{dp39}, \eqref{dp42} and \eqref{dp43}, that, for all $\phi,\ \phi_1,\ \phi_2\in B_{\varepsilon , \delta_\varepsilon (t),\xi}(\gamma)$,

$$ \left\|T_{\varepsilon , \delta_{\varepsilon}(t),\xi}(\phi)\right\|_{P_g}\leq  \left\{\begin{array}{lll} C(\gamma^{2^\ast-1-\varepsilon}\left\|R_{\varepsilon , \delta_\varepsilon (t),\xi} \right\|_{P_g}^{2^\ast-1-\varepsilon}+\left\|R_{\varepsilon , \delta_\varepsilon (t),\xi} \right\|_{P_g} )\ \mathrm{if}\ n\geq 12 \\
 C(\gamma^{2}\left\|R_{\varepsilon , \delta_\varepsilon (t),\xi} \right\|_{P_g}^{2}+\gamma^{2^\ast-1-\varepsilon}\left\|R_{\varepsilon , \delta_\varepsilon (t),\xi} \right\|_{P_g}^{2^\ast-1-\varepsilon} \\ \quad\quad\quad +\left\|R_{\varepsilon , \delta_\varepsilon (t),\xi} \right\|_{P_g} ) \ \ 
\quad\quad\quad\quad\quad  \mathrm{if}\ 5\leq n< 12 \end{array}
\right.
$$
and
$$\left\|T_{\varepsilon , \delta_{\varepsilon}(t),\xi}(\phi_1)-T_{\varepsilon , \delta_{\varepsilon}(t),\xi}(\phi_2)\right\|_{P_g}\leq C\gamma^{2^\ast-2-\varepsilon}\left\|R_{\varepsilon , \delta_\varepsilon (t),\xi} \right\|_{P_g}^{2^\ast-2-\varepsilon}\left\|\phi_1-\phi_2\right\|_{P_g} ,$$
where $C$ stands for positive constants not depending on $\gamma,\ \varepsilon,\ \xi,\ t,\ \phi,\ \phi_1$ and $\phi_2$. Thus from Lemma \ref{reste}, if $\gamma$ is fixed large enough, for $\varepsilon$ small, for any $t\in [a,b]$ and any $\xi \in M$, $T_{\varepsilon , \delta_{\varepsilon}(t),\xi}$ is a contraction mapping from $ B_{\varepsilon , \delta_\varepsilon (t),\xi}(\gamma)$ onto $ B_{\varepsilon , \delta_\varepsilon (t),\xi}(\gamma)$. Therefore, using the fixed point theorem, there exists a function $\phi_{\delta_\varepsilon (t),\xi} \in K_{\delta_{\varepsilon}(t), \xi}^\bot$ which solves equation \eqref{eq2}. Now, \eqref{eqprop4.1} follows from Lemma \ref{reste}. The fact that $\phi_{\delta_\varepsilon (t),\xi}$ is continuously differentiable with respect to $t$ and $\xi$ is standard.
\end{proof}

\section{The reduced problem.}
For $\varepsilon>0$ small enough, we defined the energy associated to \eqref{eq} by, for $u\in H^2(M)$,
$$J_\varepsilon (u)=\dfrac{1}{2}\int_M (\Delta_g u)^2 +\dfrac{1}{2}\int_M A_g(\nabla_g u,\nabla_g u)dV + \dfrac{1}{2}\int_M h u^2 dV - \int_M F_\varepsilon(u)dV,$$
where $F_\varepsilon(u)=\displaystyle\int_0^u f_\varepsilon(s)ds$. We set $I_\varepsilon(t,\xi)=J_\varepsilon (u_0-W_{\delta_\varepsilon(t),\xi}+\phi_{\delta_\varepsilon(t),\xi})$, $t\in \R^\ast_+$ and $\xi\in M$ where $\phi_{\delta_\varepsilon(t),\xi}\in K_{\delta_{\varepsilon}(t), \xi}^\bot$ is the function defined in Proposition \ref{prop4.1}. In the next proposition, we give the expansion of $I_\varepsilon$ with respect to $\varepsilon$.
\begin{prop}
\label{energiered}
Let $u_0\in C^{4,\theta}(M)$, $\theta\in (0,1)$ be a nondegenerate positive solution of \eqref{eq}. Then there exist constants $c_i(n,u_0)$, $i=2,5$ depending on $n$ and $u_0$ and $c_i(n)$, $i=1,3,4$, depending on $n$ such that
\begin{equation}
\label{fin}
I_\varepsilon(t,\xi)=c_5(n,u_0)+c_2(n,u_0)\varepsilon+c_3(n) \varepsilon \ln \varepsilon -c_4(n) \varepsilon \ln (t)+c_1(n)\varphi(\xi) \varepsilon t +o(\varepsilon)
\end{equation}
as $\varepsilon \rightarrow 0$ $C^0$ uniformly with respect to $t$ in compact subsets of $\R^\ast_+$ and with respect to $\xi \in M$ and $C^1$ uniformly if $8\leq n\leq 13$. Moreover, we have that $c_4(n)> 0$,
$c_1(n)=\dfrac{2}{n}K_n^{-\frac{n}{4}}$ and 
\begin{equation*}\begin{multlined}\varphi(\xi)=\left( \dfrac{(n-1)}{(n-6)(n^2-4)}(Tr_g (A_g- A_{paneitz})(\xi)1_{n\geq 8}\right.\\ +\left. \dfrac{2^{n} u_0(\xi)\omega_{n-1}}{(n+2)(n (n-4)(n^2-4))^{\frac{n-4}{8}} \omega_n} 1_{n\leq 8}\right),\end{multlined}\end{equation*}
where $\omega_n$ stands for the volume of $\mathbb{S}^n$ and $K_n$ is the sharp constant for the embedding of 
$H^2(\R^n)$ into $L^{2^\ast}(\R^n)$ given by $K_n^{-1}=\dfrac{n (n-4)(n^2-4)\omega_n^{\frac{4}{n}}}{16}$.

\end{prop}
\begin{proof}
We begin by proving that
\begin{equation}
\label{1305}
I_\varepsilon (t,\xi)=J_\varepsilon (u_0-W_{\delta_\varepsilon (t),\xi})+o(\varepsilon),
\end{equation}
as $\varepsilon \rightarrow 0$, uniformly with respect to $t$ in compact subsets of $\R^\ast_+$ and points $\xi\in M$ (we will show in Lemma \ref{c1est} that, when $8\leq n\leq 13$, this estimate holds $C^1$ uniformly with respect to $t$ and $\xi$). Indeed, we have
\begin{equation}\label{proofthm2e1}\begin{multlined}
I_\varepsilon (t,\xi)-J_\varepsilon (u_0-W_{\delta_\varepsilon (t),\xi})\\
=
\left\langle u_0-W_{\delta_\varepsilon (t),\xi}-i^\ast (f_\varepsilon(u_0-W_{\delta_\varepsilon (t),\xi})),\phi_{\delta_\varepsilon (t),\xi} \right\rangle_{P_g}+O(\left\|\phi_{\delta_\varepsilon (t),\xi}\right\|^2_{P_g})
\end{multlined}\end{equation}
when $\varepsilon\rightarrow 0$. Using Lemma \ref{reste} and Proposition \ref{prop4.1}, we get
\begin{equation*}\begin{multlined}
\left\langle u_0-W_{\delta_\varepsilon (t),\xi}-i^\ast (f_\varepsilon(u_0-W_{\delta_\varepsilon (t),\xi})),\phi_{\delta_\varepsilon (t),\xi} \right\rangle_{P_g}\\+ O(\left\|\phi_{\delta_\varepsilon (t),\xi}\right\|^2_{P_g})= O(\varepsilon^2 |\ln \varepsilon|^2)=o(\varepsilon).
\end{multlined}\end{equation*}
Now, the  proposition is reduced to estimate $J_\varepsilon (u_0-W_{\delta_\varepsilon (t),\xi})$. We will focus on $C^0$-estimates. The $C^1$-estimates can be obtained using the same argument as in Lemma 4.1 of \cite{MR2542977}. 
Since $u_0$ is a solution of \eqref{eq}, we have
\begin{equation*}\begin{multlined}
J_\varepsilon(u_0-W_{\delta_\varepsilon (t),\xi})
=\dfrac{1}{2}\int_M u_0^{2^\ast}dV + \dfrac{1}{2}\int_M (\Delta_g W_{\delta_\varepsilon (t),\xi})^2 dV\\+\dfrac{1}{2}\int_M A_g(\nabla_g W_{\delta_\varepsilon (t),\xi}, \nabla_g W_{\delta_\varepsilon (t),\xi})dV
+ \dfrac{1}{2}\int_M h W_{\delta_\varepsilon (t),\xi}^2 dV\\-\int_M f_\varepsilon(u_0) W_{\delta_\varepsilon (t),\xi}dV - \int_M F_\varepsilon(u_0-W_{\delta_\varepsilon (t),\xi})dV. 
\end{multlined}\end{equation*}
Using a Taylor expansion with respect to $\varepsilon$, we get
\begin{equation*}\begin{multlined}
\dfrac{1}{2}\int_M u_0^{2^\ast}dV-\dfrac{1}{2^\ast-\varepsilon}\int_M u_0^{2^\ast -\varepsilon}dV\\
= \dfrac{1}{2}\int_M u_0^{2^\ast}dV- \dfrac{1}{2^\ast}(1+\dfrac{\varepsilon}{2^\ast})\int_M u_0^{2^\ast} (1-\varepsilon \ln u_0)dV+O(\varepsilon^2)\\
= (\dfrac{1}{2}-\dfrac{1}{2^\ast})\int_M u_0^{2^\ast} dV+ \dfrac{\varepsilon}{2^\ast}\int_M u_0^{2^\ast} (\ln u_0 - \dfrac{1}{2^\ast})dV+O(\varepsilon^2)
\end{multlined}\end{equation*}
Thus from the two previous equalities, we obtain
\begin{equation}\begin{multlined}
\label{eneI}
J_\varepsilon(u_0-W_{\delta_\varepsilon (t),\xi})= (\dfrac{1}{2}-\dfrac{1}{2^\ast})\int_M u_0^{2^\ast}dV + \dfrac{\varepsilon}{2^\ast}\int_M u_0^{2^\ast} (\ln u_0 - \dfrac{1}{2^\ast})dV\\
+ I_{1,\varepsilon,t,\xi}+I_{2,\varepsilon,t,\xi}+I_{3,\varepsilon,t,\xi}+O(\varepsilon^2),
\end{multlined}\end{equation} 
where 
\begin{equation*}\begin{multlined}
I_{1,\varepsilon,t,\xi}=\dfrac{1}{2}\int_M (\Delta_g W_{\delta_\varepsilon (t),\xi})^2 dV+\dfrac{1}{2}\int_M A_g(\nabla_g W_{\delta_\varepsilon (t),\xi}, \nabla_g W_{\delta_\varepsilon (t),\xi})dV\\
+
\dfrac{1}{2}\int_M h W_{\delta_\varepsilon (t),\xi}^2dV-\int_M F_\varepsilon(W_{\delta_\varepsilon (t),\xi})dV, 
\end{multlined}\end{equation*}
\begin{flalign*}&I_{2,\varepsilon,t,\xi}=\int_M f_\varepsilon(W_{\delta_\varepsilon (t),\xi}) u_0dV,&&\end{flalign*}
and
\begin{equation}\begin{multlined}I_{3,\varepsilon,t,\xi}=-\int_M F_\varepsilon(u_0-W_{\delta_\varepsilon (t),\xi})-F_\varepsilon(u_0)-F_\varepsilon(W_{\delta_\varepsilon (t),\xi})\\+f_\varepsilon(u_0)W_{\delta_\varepsilon (t),\xi}+f_\varepsilon(W_{\delta_\varepsilon (t),\xi})u_0dV.\end{multlined}\end{equation}
We begin by estimating $I_3$. Using Taylor expansion (cf \eqref{dp2}) and rough estimations, we have
\begin{align*}
|I_{3,\varepsilon,t,\xi}| \leq{} &  \left\|(F_\varepsilon(u_0-W_{\delta_\varepsilon (t),\xi})-F_\varepsilon(W_{\delta_\varepsilon (t),\xi})+f_\varepsilon(W_{\delta_\varepsilon (t),\xi})u_0)1_{B(\sqrt{\delta_\varepsilon (t)})}\right\|_{L^1}\\
+& \left\|(F_\varepsilon(u_0-W_{\delta_\varepsilon (t),\xi})-F_\varepsilon(u_0)+f_\varepsilon(u_0)W_{\delta_\varepsilon (t),\xi})1_{M\backslash B(\sqrt{\delta_\varepsilon (t)})}\right\|_{L^1}\\
+& \left\|F_\varepsilon(u_0)1_{B(\sqrt{\delta_\varepsilon (t)})}\right\|_{L^1}+\left\|f_\varepsilon(u_0)W_{\delta_\varepsilon (t),\xi} 1_{B(\sqrt{\delta_\varepsilon (t)})}\right\|_{L^1}\\
+&\left\|F_\varepsilon(W_{\delta_\varepsilon (t),\xi})1_{M\backslash B(\sqrt{\delta_\varepsilon (t)})}\right\|_{L^1}+\left\|u_0 f_\varepsilon(W_{\delta_\varepsilon (t),\xi})1_{M\backslash B(\sqrt{\delta_\varepsilon (t)})}\right\|_{L^1}\\
\leq{} &  \left\|u_0^2 W_{\delta_\varepsilon (t),\xi}^{2^\ast-2-\varepsilon}1_{B(\sqrt{\delta_\varepsilon (t)})}\right\|_{L^1}+\left\|u_0^{2^\ast -2-\varepsilon}W_{\delta_\varepsilon (t),\xi}^2 1_{M\backslash B(\sqrt{\delta_\varepsilon (t)})}\right\|_{L^1}\\
+&\left\|F_\varepsilon(W_{\delta_\varepsilon (t),\xi})1_{M\backslash B(\sqrt{\delta_\varepsilon (t)})}\right\|_{L^1}+\left\|u_0 f_\varepsilon(W_{\delta_\varepsilon (t),\xi})1_{M\backslash B(\sqrt{\delta_\varepsilon (t)})}\right\|_{L^1}\\
+& \left\|F_\varepsilon(u_0)1_{B(\sqrt{\delta_\varepsilon (t)})}\right\|_{L^1}+\left\|f_\varepsilon(u_0)W_{\delta_\varepsilon (t),\xi}1_{B(\sqrt{\delta_\varepsilon (t)})}\right\|_{L^1}\\
\leq{} &  C\left\|u_0^2 W_{\delta_\varepsilon (t),\xi}^{2^\ast-2-\varepsilon}1_{B(\sqrt{\delta_\varepsilon (t)})}\right\|_{L^1}\nonumber+ C\left\|u_0^{2^\ast -2-\varepsilon}W_{\delta_\varepsilon (t),\xi}^2 1_{M\backslash B(\sqrt{\delta_\varepsilon (t)})}\right\|_{L^1}\\
+& O(\delta_\varepsilon (t)^{\frac{n}{2}})
\end{align*}
 Therefore estimating the last two terms and using the definition of $\delta$, we obtain
\begin{equation}
\label{enei1}
|I_{3,\varepsilon,t,\xi}|\leq \left\{\begin{array}{lll}
O(\delta_\varepsilon (t)^{\frac{n}{2}})=O(\varepsilon^{\frac{n}{4}})=o(\varepsilon^2)\ \mathrm{if}\ n>8 \\
O(\delta_\varepsilon (t)^4 |\ln \delta|)=O(\varepsilon^2 |\ln \varepsilon|) \ \mathrm{if}\ n=8\\
O(\delta_\varepsilon (t)^{n-4})=O(\varepsilon^2) \ \mathrm{if}\ n<8.
\end{array}
\right.
\end{equation}
Now, let us estimate $I_{2,\varepsilon,t,\xi}$. We recall that the Cartan expansion of the metric gives
\begin{equation}\label{g}\sqrt{|g|}(x)=1-\dfrac{1}{6}Ric_{ij}x^i x^j-\dfrac{1}{12}\nabla_kRic_{ij}x^ix^jx^k+O(|x|^4),\end{equation}
where $|g|$ stands for the determinant of the metric $g$ in geodesic normal coordinates. Then, using a change of variables, Taylor expansion and by symmetry, we have
\begin{align}
\label{enei2}
 I_{2,\varepsilon,t,\xi} = {} & u_0 (\xi) \omega_{n-1} \alpha_n^{\frac{n+4}{n-4}-\varepsilon}\delta_\varepsilon (t)^{\frac{n-4}{2}(1+\varepsilon)}\nonumber\\
&  \times\int_0^{\frac{r_0}{2\delta_\varepsilon}(t)} 
\dfrac{r^{n-1}}{(1+r^2)^{\frac{n+4}{2}-\varepsilon \frac{n-4}{2}}}(1+O(\delta^2r^2))dr\nonumber\\
& +  O(\delta_\varepsilon(t)^{\frac{n}{2}}+\varepsilon^2 |\ln \delta_\varepsilon (t)|)\nonumber\\
= {}& \dfrac{2 u_0(\xi)\omega_{n-1}\alpha_n^{\frac{n+4}{n-4}}\delta_\varepsilon (t)^{\frac{n-4}{2}} }{n(n+2)}+O(\delta_\varepsilon (t)^{\frac{n}{2}}+\varepsilon^2 |\ln \delta_\varepsilon (t)|)\nonumber\\
= {}&  \dfrac{2^{n+1}u_0(\xi) K_n^{-\frac{n}{4}}\omega_{n-1}\delta_\varepsilon (t)^{\frac{n-4}{2}}}{n(n+2)\alpha_n \omega_n}+O(\delta_\varepsilon (t)^{\frac{n}{2}}+\varepsilon^2 |\ln \delta_\varepsilon (t)|),
\end{align}
where $\alpha_n$ is defined in \eqref{defalpha}. 
Finally, we use the computations of section $4$ of \cite{MR1942129} and the estimate $(4.2)$ of \cite{MR2859126} to estimate $I_{1,\varepsilon,t,\delta}$. 
We notice, using \eqref{g} and by symmetry, that the remaining in equation $(4.2)$ of \cite{MR2859126} (namely $o(\delta_\varepsilon (t)^2)$ ) is actually in $O(\delta_\varepsilon(t)^4)$. 
We thus have
\begin{equation}\begin{multlined}
\label{enei3}
I_{1,\varepsilon,t,\delta}= \dfrac{2}{n}K_n^{-\frac{n}{4}}\left(1-C_n \varepsilon-\dfrac{(n-4)^2}{8}\varepsilon \ln \delta\right. 
\\ +\dfrac{(n-1)}{(n-6)(n^2-4)}(Tr_g (A_g- A_{paneitz}) \delta_\varepsilon (t)^2  1_{n\geq 8})\\
+ o(\varepsilon)+O(\delta_\varepsilon (t)^4)\bigg),
\end{multlined}\end{equation}
where
\begin{equation}\begin{multlined}C_n=2^{n-4}(n-4)^2 \dfrac{\omega_{n-1}}{\omega_n}\int_0^\infty \dfrac{r^{\frac{n-2}{2}}\ln (1+r)}{(1+r)^n}dr\\ +\dfrac{(n-4)^2}{8(n-2)}(1-\dfrac{1}{2}\ln \sqrt{n(n-4)(n^2-4)}).\end{multlined}\end{equation}
Thus, combining \eqref{eneI}, \eqref{enei1}, \eqref{enei2} and \eqref{enei3}, we obtain

\begin{align}
\label{13051}
J_\varepsilon(u_0-W_{\delta_\varepsilon (t),\xi})
=  {} &(\dfrac{1}{2}-\dfrac{1}{2^\ast})\int_M u_0^{2^\ast} + \dfrac{\varepsilon}{2^\ast}\int_M u_0^{2^\ast} (\ln u_0 - \dfrac{1}{2^\ast})dV\nonumber\\
 &+\dfrac{2}{n}K_n^{-\frac{n}{4}}\left(1-C_n \varepsilon-\dfrac{(n-4)^2}{8}\varepsilon \ln \delta_\varepsilon (t) \right.\nonumber\\
 & \quad+ \left. \dfrac{(n-1)}{(n-6)(n^2-4)}(Tr_g (A_g- A_{paneitz}) \delta_\varepsilon (t)^2\right)\nonumber\\
&+ \dfrac{2^{n+1}u_0(\xi) K_n^{-\frac{n}{4}}\omega_{n-1}\delta_\varepsilon (t)^{\frac{n-4}{2}}}{n(n+2)\alpha_n \omega_n}+o(\varepsilon).
\end{align}
The lemma follows from \eqref{1305} and \eqref{13051}.
\end{proof}

The next proposition shows that, in order to construct a solution to \eqref{eq}, we only need to find a critical point for the reduced energy $I_\varepsilon$. 
\begin{prop}
\label{propfin}
 
Given two positive real numbers $a<b$, for $\varepsilon$ small, if $(t_\varepsilon,\xi_\varepsilon)\in (a,b)\times M$ is a critical point of $I_\varepsilon$, then the function
$u_0-W_{\delta_{\varepsilon}(t_\varepsilon),\xi_\varepsilon}+\phi_{\delta_{\varepsilon}(t_\varepsilon),\xi_\varepsilon}$ is a solution of \eqref{eq}.
\end{prop}
\begin{proof}
Let $(\xi_\alpha)_\alpha$ be a sequence of points of $M$ and suppose that $(t_\alpha)_\alpha$ and $(\varepsilon_\alpha)_\alpha$ are two sequences of real numbers such that $\varepsilon_\alpha \underset{\alpha\rightarrow \infty}{\longrightarrow} 0$, $a\leq t_\alpha \leq b$ and $(t_\alpha,\xi_\alpha)$ is a critical point of $I_{\varepsilon_\alpha}$ for all $\alpha \in \N$. To simplify notations, we set, for $i=1,\ldots ,n$,
$$Z_{0,\alpha}= Z_{\delta_{\varepsilon_\alpha}(t_\alpha),\xi_\alpha}\ and\ Z_{i,\alpha}=Z_{\delta_{\varepsilon_\alpha}(t_\alpha),\xi_\alpha,e_i}.$$
Since $\phi_{\delta_{\varepsilon_\alpha}(t_\alpha),\xi_\alpha}$ is a solution of \eqref{eq2} by Proposition \ref{prop4.1}, there exist real numbers $\lambda_{i,\alpha}$, $i=0,\ldots , n$ such that 
\begin{equation}
\label{rob69}
DJ_{\varepsilon_\alpha} (u_0-W_{\delta_{\varepsilon_\alpha}(t_\alpha),\xi_\alpha}+\phi_{\delta_{\varepsilon_\alpha}(t_\alpha),\xi_\alpha})=\sum_{i=0}^n \lambda_{i,\alpha} \left\langle Z_{i,\alpha}, .\right\rangle_{P_g}.
\end{equation}
Using the previous equality, we see that
\begin{equation}
\label{rob65}
\dfrac{\partial I_{\varepsilon_\alpha}}{\partial t}(t_\alpha,\xi_\alpha)=\sum_{i=0}^n \lambda_{i,\alpha} \left\langle Z_{i,\alpha},\dfrac{\partial}{\partial t}(-W_{\delta_{\varepsilon_\alpha}(t_\alpha),\xi_\alpha}+\phi_{\delta_{\varepsilon_\alpha}(t_\alpha),\xi_\alpha}) \right\rangle_{P_g}.
\end{equation}
A simple computation gives
\begin{equation}
\label{rob66}
\dfrac{\partial}{\partial t}(W_{\delta_{\varepsilon_\alpha}(t_\alpha),\xi_\alpha})|_{t=t_\alpha}=\dfrac{\tilde{C}_n}{t_\alpha}Z_{0,\alpha},
\end{equation}
where $\tilde{C}_n=\alpha_n$ if $n<8$ and $\tilde{C}_n=\dfrac{\alpha_n (n-4)}{4}$ if $n\geq 8$ (see \eqref{defalpha} for the definition of $\alpha_n$).
Taking the derivative of $\left\langle Z_{\delta_{\varepsilon_\alpha}(t_\alpha),\xi_\alpha},\phi_{\delta_{\varepsilon_\alpha}(t_\alpha),\xi_\alpha}\right\rangle_{P_g}=0$ with respect to $t$, we obtain
\begin{equation}
\label{rob67}
\left\langle \dfrac{\partial}{\partial t}Z_{\delta_{\varepsilon_\alpha}(t_\alpha)|_{t=t_\alpha},\xi_\alpha},\phi_{\delta_{\varepsilon_\alpha}(t_\alpha),\xi_\alpha}\right\rangle_{P_g}=-\left\langle Z_{\delta_{\varepsilon_\alpha}(t_\alpha),\xi_\alpha},\dfrac{\partial}{\partial t}\phi_{\delta_{\varepsilon_\alpha}(t_\alpha),\xi_\alpha}|_{t=t_\alpha}\right\rangle_{P_g}.
\end{equation}
Since a straight forward computation gives $\left\|\dfrac{\partial}{\partial t}Z_{\delta_{\varepsilon_\alpha}(t_\alpha),\xi_\alpha}|_{t=t_\alpha} \right\|_{P_g}=O(1) $, from \eqref{eqprop4.1}, \eqref{rob65}, \eqref{rob66} and \eqref{rob67}, we deduce that
\begin{equation}
\label{rob70}
\dfrac{\partial I_{\varepsilon_\alpha}}{\partial t}(t_\alpha,\xi_\alpha)= -\dfrac{\tilde{C}_n}{t_\alpha}\lambda_{0,\alpha} \left\|\Delta_{eucl} V_0\right\|_{L^2(\R^n)}^2+o(\sum_{i=0}^n \lambda_{i,\alpha} ),
\end{equation}
where $o(1) \underset{\alpha \rightarrow +\infty}{\longrightarrow} 0$. Arguing the same way and noting that
$$\dfrac{\partial}{\partial y_i}(W_{\delta_{\varepsilon_\alpha}(t_\alpha),\exp_{\xi_\alpha}(y)})|_{y=0}=\dfrac{\alpha_n (n-4)}{\delta_{\varepsilon_\alpha}(t_\alpha)}Z_{i,\alpha}+R_{i,\alpha}$$
where $R_{i,\alpha}\underset{\alpha\rightarrow +\infty}{\longrightarrow} 0$ in $H^2(M)$, and

$$\left\|\dfrac{\partial}{\partial y_i}Z_{j,\delta_{\varepsilon_\alpha}(t_\alpha),\exp_{\xi_\alpha}(y)}|_{y=0} \right\|_{P_g}=O(\dfrac{1}{\delta_{\varepsilon_\alpha}(t_\alpha)}) ,$$
we obtain
\begin{equation}
\label{rob71}
\delta_{\varepsilon_\alpha}(t_\alpha)\dfrac{\partial I_{\varepsilon_\alpha}}{\partial y_i}(t_\alpha,\exp_{\xi_\alpha}(y))|_{y=0}= -\lambda_{i,\alpha} \left\|\Delta_{eucl} V_i\right\|_{L^2(\R^n)}^2+o(\sum_{i=0}^n \lambda_{i,\alpha} ). 
\end{equation}
Therefore, from \eqref{rob69}, \eqref{rob70} and \eqref{rob71}, it follows that if $(t_\alpha ,\xi_\alpha)$ is a critical point of $I_{\varepsilon_\alpha}$ then $u_0-W_{\delta_{\varepsilon_\alpha}(t_\alpha),\xi_\alpha}+\phi_{\delta_{\varepsilon_\alpha}(t_\alpha),\xi_\alpha}$ is a solution of \eqref{eq}.
\end{proof}
We are now in position to prove the theorems.

\section{Proof of the theorems.}
We begin by proving Theorem \ref{thm1}.

\begin{proof}[Proof of Theorem \ref{thm1}.]
We set $\textit{G}: \R^\ast_+ \times M\rightarrow \R$ the function defined by
$$\textit{G}(t,\xi)=-c_4(n)\ln t+c_1(n)\varphi(\xi)t,$$
where $c_4(n),\ c_1(n)\ \mathrm{and}\ \varphi(\xi)$ are defined in \eqref{fin}. From Proposition \ref{energiered}, we have
\begin{equation}\label{convergence}\lim_{\varepsilon \rightarrow 0}\dfrac{1}{\varepsilon}(I_\varepsilon (t,\xi)-c_5(n,u_0)-c_2(n,u_0)\varepsilon-c_3(n) \varepsilon \ln \varepsilon)=\textit{G}(t,\xi),\end{equation}
$C^1$ uniformly with respect to $\xi\in M$ and $t$ in compact subset of $\R^\ast_+$. We will consider two cases depending on the dimension of the manifold.\\
\\
\textbf{First case : $8\leq n\leq 13$.}\\

 We argue as in \cite{MR2542977}. Let $\xi_0$ be the $C^1$ stable critical point of $\varphi$ such that $\varphi(\xi_0)>0$
 and set $$t_0=\dfrac{c_4(n)}{c_1(n)\varphi(\xi_0)}>0.$$
Identifying the tangent space at $\xi$ with $\R^n$ we define the map $H$ from $[0,1]\times \R^+ \times \R^n$
into  $\R^{n+1}$ by
\begin{equation*}\begin{multlined}
H(s,t,\xi)= s \left(\dfrac{\partial G(t,\exp_\xi(y))}{\partial t}, \dfrac{\partial G(t,\exp_\xi(y))}{\partial y_1}|_{y=0},\ldots, \dfrac{\partial G(t,\exp_\xi(y))}{\partial y_n}|_{y=0}\right)\\
\!+\! (1-s) \left(t-t_0, \dfrac{\partial (\varphi \circ \exp_\xi(y))}{\partial y_1}|_{y=0},\ldots,
\dfrac{\partial (\varphi \circ \exp_\xi(y))}{\partial y_n}|_{y=0}\right)\!.
\end{multlined}\end{equation*}
 By the invariance of the Brower degree via homotopy, we have that $(t_0,\xi_0)$ is a $C^1$
 stable critical point of $G$. From Proposition \ref{energiered} and standard properties of the Brower degree (see \emph{e.g.} \cite{MR1373430}), there exists a couple $(t_\varepsilon ,\xi_\varepsilon)$
 of critical points of $I_\varepsilon$ converging to $(t_0,\xi_0)$.\\
\\
\textbf{Second case : $5\leq n< 8$ and $n>13$.}\\

 Since $c_4(n)$ and $c_1(n)$ are positive, we have
$$\lim_{t\rightarrow 0^+}\textit{G}(t,\xi)=\lim_{t\rightarrow \infty}\textit{G}(t,\xi)=+\infty,$$
uniformly in $\xi \in M$. Therefore, from \eqref{convergence}  we deduce that, for $\varepsilon$ small enough,there exists a couple $(t_\varepsilon ,\xi_\varepsilon)$ 
which is a minimum for the functional $I_\varepsilon$ in $(a,b)\times M$ where $a,b$ are positive constants not depending on $\varepsilon$.
This implies from Proposition \ref{propfin} that $u_0-W_{\delta_{\varepsilon}(t_\varepsilon),\xi_\varepsilon}-\phi_{\delta_{\varepsilon}(t_\varepsilon),\xi_\varepsilon}$ 
is a solution of \eqref{eq}. Thus Theorem \ref{thm1} is established.
\end{proof}

Finally, we prove Theorem \ref{thm2}.
\begin{proof}[Proof of Theorem \ref{thm2}.]
The proof of Theorem \ref{thm2} will follow closely the proof of Theorem \ref{thm1} therefore we will only sketch it. We restrict ourselves to the case where $9\leq n \leq 11$ (the case $5\leq n\leq 8$ is contained in Theorem \ref{thm1}). The main difference is that here we will take $\delta_\varepsilon (t_\varepsilon)= (t_\varepsilon \varepsilon)^{\frac{2}{n-4}}$, for $9\leq n\leq 11$. We will only point out the impact of this choice in the two key estimates, namely the estimate of $\phi_{\delta_\varepsilon (t),\xi}$ in Proposition \ref{prop4.1} (given in Lemma \ref{reste}) and the estimate of the reduced energy (see Proposition \ref{energiered}). Let us first consider the error estimate i.e. Lemma \ref{reste}. With our new choice of $\delta (t_\varepsilon)$, it is immediate to check that the leading term in the expansion of Lemma \ref{reste} will be given by the term $\|f_0 (W_{\delta_\varepsilon (t),\xi})-P(W_{\delta_\varepsilon (t),\xi}) \|_{L^{\frac{2n}{n+4}}} $. This implies that Lemma \ref{reste} will rewrite as 
\begin{equation}
\label{newreste}
\| i^\ast (f_\varepsilon(u_0-W_{\delta_\varepsilon (t),\xi}))-u_0+W_{\delta_\varepsilon (t),\xi} \|_{P_g} =0( \delta_\varepsilon (t)^2)=0( \varepsilon^{\frac{4}{n-4}}).
\end{equation}
Therefore we deduce that 
\begin{equation}
\label{newestphi}
\|\phi_{\delta_\varepsilon (t),\xi}\|_{P_g}=0( \varepsilon^{\frac{4}{n-4}}),
\end{equation} 
where $\phi_{\delta_\varepsilon (t),\xi}$ is the function defined in Proposition \ref{prop4.1}. Now, let us consider the changes that occur in Proposition \ref{energiered}. Using \eqref{proofthm2e1}, \eqref{newreste} and \eqref{newestphi}, we obtain that, for $9\leq n\leq 11$,
$$I_\varepsilon (t,\xi)-J_\varepsilon (u_0 -W_{\delta_\varepsilon (t),\xi})=0(\|\phi_{\delta_\varepsilon (t),\xi}\|_{P_g}^2 )  0( \delta_\varepsilon^4 (t))=0(\varepsilon^{\frac{8}{n-4}})=o(\varepsilon).$$
Then, it only remains to compute $J_\varepsilon (u_0 -W_{\delta_\varepsilon (t),\xi})$. Being a bit careful with the different remainings apppearing in the proof of Proposition \ref{energiered} and using that $A_g=A_{paneitz}$, we see that
\begin{align*}
J_\varepsilon(u_0-W_{\delta_\varepsilon (t),\xi})
=  {} &(\dfrac{1}{2}-\dfrac{1}{2^\ast})\int_M u_0^{2^\ast} + \dfrac{\varepsilon}{2^\ast}\int_M u_0^{2^\ast} (\ln u_0 - \dfrac{1}{2^\ast})dV\nonumber\\
 &+\dfrac{2}{n}K_n^{-\frac{n}{4}}\left(1-C_n \varepsilon-\dfrac{(n-4)}{4}\varepsilon \ln (t\varepsilon ) \right)\nonumber\\
&+ \dfrac{2^{n+1}u_0(\xi) K_n^{-\frac{n}{4}}\omega_{n-1} t\varepsilon}{n(n+2)\alpha_n \omega_n}+o(\varepsilon).
\end{align*}
Using this last estimate, we can argue exactly as in the case $5\leq n<8$ of the proof of Theorem \ref{thm1}. This concludes the proof of Theorem \ref{thm2}. 
\end{proof}

\section{Appendix.}
In this section, we will give an estimate of the error $R_{\varepsilon , \delta_\varepsilon (t),\xi}$ (see Proposition \ref{prop4.1}) and complete the proof of Proposition \ref{energiered} by showing that \eqref{1305} holds $C^1$ uniformly with respect to $t$ in compact subsets of $\R^\ast_+$ and $\xi \in M$ when $8\leq n\leq 13$. Let us begin with the estimate of the error.
\begin{lem}
\label{reste}
Given two positive real numbers $a<b$, there exists a positive constant $C^\prime_{a,b}$ such that for $\varepsilon$ small, for any real number $t\in [a,b]$ and any point $\xi \in M$, there holds
 $$\left\|i^\ast (f_\varepsilon (u_0-W_{\delta_\varepsilon (t),\xi}))-u_0+W_{\delta_\varepsilon (t) , \xi}\right\|_{P_g} \leq C^\prime_{a,b}\varepsilon |\ln \varepsilon |$$

\end{lem}
\begin{proof} 

All the estimates will be uniform in $t,\xi$ and $\varepsilon$. Since $i^\ast$ is continuous, we have
\begin{equation}\begin{multlined}\left\|i^\ast (f_\varepsilon (u_0-W_{\delta_\varepsilon (t),\xi}))-u_0+W_{\delta_\varepsilon (t) , \xi}  \right\|_{P_g}\\= O\left(\left\| (f_\varepsilon (u_0-W_{\delta_\varepsilon (t),\xi}))-P_g(u_0-W_{\delta_\varepsilon (t) , \xi})\right\|_{L^\frac{2n}{n+4}}\right)
\end{multlined}\end{equation}
where $f_\varepsilon (u)=|u|^{2^\ast-2-\varepsilon}u$.
The triangular inequality yields to
\begin{align} 
\lefteqn{\left\|i^\ast (f_\varepsilon (u_0-W_{\delta_\varepsilon (t),\xi}))-u_0+W_{\delta_\varepsilon (t) , \xi}  \right\|_{P_g}}\nonumber\\
& \quad\quad\quad\quad\quad\leq  C  \left\|f_\varepsilon (u_0-W_{\delta_\varepsilon (t),\xi})-f_\varepsilon (u_0)+f_\varepsilon (W_{\delta_\varepsilon (t),\xi})\right\|_{L^\frac{2n}{n+4}}\nonumber\\
&\quad\quad\quad\quad\quad\ \ + C  \left\|f_\varepsilon (u_0)-P_g (u_0)\right\|_{L^\frac{2n}{n+4}}\nonumber\\
&\quad\quad\quad\quad\quad\ \ +  C  \left\|f_\varepsilon (W_{\delta_\varepsilon (t) ,\xi})-P_g (W_{\delta_\varepsilon (t),\xi})\right\|_{L^\frac{2n}{n+4}}\nonumber\\
{} &\quad\quad\quad\quad\quad\leq C (I_1 +I_2+I_3).
\end{align}

We first estimate $I_1$. By triangular inequality we get

\begin{align}
I_1\leq{} &  \left\|(f_\varepsilon (u_0-W_{\delta_\varepsilon (t),\xi})+f_\varepsilon (W_{\delta_\varepsilon (t),\xi}) ) 1_{B_\xi (\sqrt{\delta_\varepsilon (t)})} \right\|_{L^\frac{2n}{n+4}}\nonumber\\
&+\left\|(f_\varepsilon (u_0-W_{\delta_\varepsilon (t),\xi})-f_\varepsilon (u_0))1_{M\backslash B_{\xi}(\sqrt{\delta_\varepsilon (t)})}\right\|_{L^\frac{2n}{n+4}}\nonumber\\
&+\left\|f_\varepsilon (W_{\delta_\varepsilon (t),\xi})1_{M\backslash B_{\xi}(\sqrt{\delta_\varepsilon (t)})} \right\|_{L^\frac{2n}{n+4}} +   \left\|f_\varepsilon (u_0)1_{ B_{\xi}(\sqrt{\delta_\varepsilon (t)})}  \right\|_{L^\frac{2n}{n+4}}.
\end{align}
From Taylor expansion (\emph{e.g.} using \eqref{dp2}) and Young inequality,  we obtain 
\begin{equation*}\begin{multlined}
\left\|(f_\varepsilon (u_0-W_{\delta_\varepsilon (t),\xi})+f_\varepsilon (W_{\delta_\varepsilon (t),\xi}) ) 1_{B_\xi (\sqrt{\delta_\varepsilon (t)})} \right\|_{L^\frac{2n}{n+4}}\\
\leq  C \left\|u_0 W_{\delta_\varepsilon (t),\xi}^{2^\ast -2-\varepsilon} 1_{B_\xi (\sqrt{\delta_\varepsilon (t)})} \right\|_{L^\frac{2n}{n+4}}
+C\left\|u_0^{2^\ast-1-\varepsilon}1_{B_\xi (\sqrt{\delta_\varepsilon (t)})} \right\|_{L^\frac{2n}{n+4}},
\end{multlined}\end{equation*}
as well as 
\begin{equation*}
\begin{multlined}
\left\|(f_\varepsilon (u_0-W_{\delta_\varepsilon (t),\xi})-f_\varepsilon (u_0))1_{M\backslash B_{\xi}(\sqrt{\delta_\varepsilon (t)})}\right\|_{L^\frac{2n}{n+4}}\\
\leq  C\left\|u_0^{2^\ast-2-\varepsilon}W_{\delta_\varepsilon (t),\xi}1_{M\backslash B_{\xi}(\sqrt{\delta_\varepsilon (t)})}\right\|_{L^\frac{2n}{n+4}}+C\left\|W_{\delta_\varepsilon (t),\xi}^{2^\ast-1-\varepsilon}1_{M\backslash B_{\xi}(\sqrt{\delta_\varepsilon (t)})}\right\|_{L^\frac{2n}{n+4}}.
\end{multlined}\end{equation*}

\noindent Using polar coordinates and a change of variables we deduce that:
$$I_1 =\left\{\begin{array}{lc}
O(\delta_\varepsilon^{\frac{n+4}{4}}(t))=O(\varepsilon^{\frac{n+4}{8}})\ \quad& \mathrm{if}\quad n>12, \\
O(\delta^4_\varepsilon (t) |\ln \delta_\varepsilon (t)|^{\frac{2}{3}})=O(\varepsilon^2 |\ln \varepsilon|^{\frac{2}{3}})\ \quad& \mathrm{if}\quad n=12,\\
O(\delta^{\frac{n-4}{2}}_\varepsilon (t))=O(\varepsilon)\ \quad& \mathrm{if}\quad n<12.
\end{array}
\right.
$$
Concerning $I_2$ we easily get from Taylor's expansion that
$$I_2=\left\|f_\varepsilon (u_0)-f_0(u_0) \right\|_{L^\frac{2n}{n+4}}=O(\varepsilon).$$
We now estimate $I_3$.
First we recall that with the help of the exponential map we can identify $B_\xi(R_0)$ with a neighborhood of the origin in $\R^n$.
Therefore with this chart  we may define  $\chi_{\xi,\delta_\varepsilon (t)}(.):=\chi ( d(.\delta_\varepsilon (t),\xi))$. Using triangular inequality and a change of variables, we then get 
\begin{equation*}\begin{multlined}
I_3 \leq  C\delta^{\frac{n-4}{2}\varepsilon}\left\|\chi_{\xi,\delta_\varepsilon (t)}^{2^\ast-1-\varepsilon}(U^{2^\ast-1-\varepsilon}-U^{2^\ast-1} )\right\|_{L^\frac{2n}{n+4}}\\
+C\left\|(\delta^{\frac{n-4}{2}\varepsilon}\chi_{\xi,\delta_\varepsilon (t)}^{2^\ast-1-\varepsilon}
-\chi_{\xi,\delta_\varepsilon (t)}^{2^\ast-1})U^{2^\ast -1}_{\delta_\varepsilon (t) ,\xi}  \right\|_{L^\frac{2n}{n+4}}\\
+ \left\|f_0 (W_{\delta_\varepsilon (t),\xi})-P_g(W_{\delta_\varepsilon (t),\xi}) \right\|_{L^\frac{2n}{n+4}}.
\end{multlined}\end{equation*}
Following the computation in the proof of lemma 2.3 of \cite{MR2859126} we obtain these three estimates: 
\[\left\|\chi_{\xi,\delta_\varepsilon (t)}^{2^\ast-1-\varepsilon}(U^{2^\ast-1-\varepsilon}-U^{2^\ast-1} )\right\|_{L^\frac{2n}{n+4}}=O(\varepsilon),\]

\[\left\|(\delta^{\frac{n-4}{2}}_\varepsilon (t)\chi_{\xi,\delta_\varepsilon (t)}^{2^\ast-1-\varepsilon}-\chi_{\xi,\delta_\varepsilon (t)}^{2^\ast-1})U^{2^\ast -1}  \right\|_{L^\frac{2n}{n+4}}=O(\varepsilon|\ln \delta_\varepsilon (t)|),\]
and 
$$\left\|f_0 (W_{\delta_\varepsilon (t),\xi})-P_g(W_{\delta_\varepsilon (t),\xi}) \right\|_{L^\frac{2n}{n+4}}\leq C\left\{\begin{array}{lll}
\delta^2_\varepsilon (t)=O(\varepsilon)\ \mathrm{if}\ n>8 ,\\
\delta^2_\varepsilon (t) |\ln \delta_\varepsilon (t)|=O(\varepsilon |\ln \varepsilon |)\ \mathrm{if}\ n=8,\\
\delta^{\frac{n-4}{2}}_\varepsilon (t)=O(\varepsilon) \ \mathrm{if}\ n<8.
\end{array}
\right.$$ 
This concludes the proof.

\end{proof} 
 Finally, let us prove that \eqref{1305} holds $C^1$ uniformly with respect to $t$ in compact subsets of $\R^\ast_+$ and $\xi \in M$ when $8\leq n\leq 13$.
\begin{lem}
\label{c1est}
If $8\leq n\leq 13$, we have
$$I_\varepsilon (t,\xi)=J_\varepsilon (u_0-W_{\delta_\varepsilon (t),\xi})+o(\varepsilon)$$
$C^1$ uniformly with respect to $t$ in compact subsets of $\R^\ast_+$ and $\xi \in M$.
\end{lem}
\begin{proof}
To simplify notations, we set, for $i=1,\ldots ,n$,
$$Z_{0}= Z_{\delta_{\varepsilon}(t),\xi}\ and\ Z_{i}=Z_{\delta_{\varepsilon}(t),\xi,e_i}.$$
We recall that
$$\dfrac{\partial}{\partial t}(W_{\delta_{\varepsilon}(t),\xi})=\dfrac{\tilde{C}_n}{t}Z_{0},$$
where $\tilde{C}_n=\dfrac{\alpha_n (n-4)}{4}$ (see \eqref{defalpha} for the definition of $\alpha_n$). Taking the derivative with respect to $t$ to $I_\varepsilon(t,\xi)-J_\varepsilon (u_0-W_{\delta_\varepsilon (t),\xi})$, we obtain
\begin{align}
\label{1505e1}
\lefteqn{\dfrac{\partial I_\varepsilon}{\partial t}(t,\xi)-\dfrac{\partial J_\varepsilon}{\partial t}(u_0-W_{\delta (t),\xi})}\nonumber\\
&= \int_M P_g (\phi_{\delta_{\varepsilon}(t),\xi}) \dfrac{\partial}{\partial t}W_{\delta_{\varepsilon}(t),\xi} dV\nonumber\\
& \quad- \int_M (f_\varepsilon(u_0-W_{\delta_{\varepsilon}(t),\xi}+\phi_{\delta_{\varepsilon}(t),\xi})-f_\varepsilon(u_0-W_{\delta_{\varepsilon}(t),\xi}))\dfrac{\partial W_{\delta_{\varepsilon}(t),\xi}}{\partial t}dV\nonumber\\
&\quad+ DJ_\varepsilon(u_0 -W_{\delta_{\varepsilon}(t),\xi}+\phi_{\delta_{\varepsilon}(t),\xi})[\dfrac{\partial \phi_{\delta_{\varepsilon}(t),\xi}}{\partial t}]\nonumber\\
&= \dfrac{\tilde{C}_n}{t}\left(\int_M (P_g (Z_0) - f_\varepsilon^\prime (u_0-W_{\delta_{\varepsilon}(t),\xi})Z_0)\phi_{\delta_{\varepsilon}(t),\xi} dV \right. \nonumber\\
&\quad\quad\quad\quad\quad\quad-\left. \int_M \big(f_\varepsilon(u_0-W_{\delta_{\varepsilon}(t),\xi}+\phi_{\delta_{\varepsilon}(t),\xi})-f_\varepsilon(u_0-W_{\delta_{\varepsilon}(t),\xi})\right.\nonumber\\
&\quad\quad\quad\quad\quad\quad\quad\quad\quad-f_\varepsilon^\prime (u_0-W_{\delta_{\varepsilon}(t),\xi})\phi_{\delta_{\varepsilon}(t),\xi}\big)Z_0dV\bigg)\nonumber\\
&\quad+DJ_\varepsilon(u_0 -W_{\delta_{\varepsilon}(t),\xi}+\phi_{\delta_{\varepsilon}(t),\xi})[\dfrac{\partial \phi_{\delta_{\varepsilon}(t),\xi}}{\partial t}]\nonumber\\
&= I_1+I_2+I_3,
\end{align}
where
\begin{align}
I_1&= \dfrac{\tilde{C}_n}{t}\int_M (P_g (Z_0) - f_\varepsilon^\prime (u_0-W_{\delta_{\varepsilon}(t),\xi})Z_0)\phi_{\delta_{\varepsilon}(t),\xi} dV,\\
I_2&=-\dfrac{\tilde{C}_n}{t}\int_M (f_\varepsilon(u_0-W_{\delta_{\varepsilon}(t),\xi}+\phi_{\delta_{\varepsilon}(t),\xi})-f_\varepsilon(u_0-W_{\delta_{\varepsilon}(t),\xi})\nonumber\\
&\quad-f_\varepsilon^\prime (u_0-W_{\delta_{\varepsilon}(t),\xi})\phi_{\delta_{\varepsilon}(t),\xi})Z_0dV,\\
I_3&=DJ_\varepsilon(u_0 -W_{\delta_{\varepsilon}(t),\xi}+\phi_{\delta_{\varepsilon}(t),\xi})[\dfrac{\partial \phi_{\delta_{\varepsilon}(t),\xi}}{\partial t}].\end{align}
In the same way, recalling that
$$\dfrac{\partial}{\partial y_i}(W_{\delta_{\varepsilon}(t),\exp_\xi (y)})|_{y=0}=\dfrac{\alpha_n (n-4)}{\delta_{\varepsilon}(t)}Z_{i}+R_{\delta_\varepsilon (t),\xi},$$
where $\left\|R_{\delta_\varepsilon (t),\xi} \right\|_{P_g}= O(\delta_\varepsilon (t)^2) $ (see (6.13) of \cite{micheletti2009role}) and using \eqref{eqprop4.1}, we find
\begin{align}
\label{1505e2}
\lefteqn{\dfrac{\partial I_\varepsilon}{\partial y_i}(t,\exp_\xi (y))|_{y=0}-\dfrac{\partial J_\varepsilon}{\partial y_i}(u_0-W_{\delta (t),\exp_\xi (y)})|_{y=0}} {}&\nonumber\\
 &{}=\dfrac{\alpha_n (n-4)}{\delta_{\varepsilon}(t)}\left(\int_M (P_g (Z_i) - f_\varepsilon^\prime (u_0-W_{\delta_{\varepsilon}(t),\xi})Z_i)\phi_{\delta_{\varepsilon}(t),\xi} dV \right. \nonumber\\
&\quad\quad \quad\quad\quad\quad- \int_M \big(f_\varepsilon(u_0-W_{\delta_{\varepsilon}(t),\xi}+\phi_{\delta_{\varepsilon}(t),\xi})-f_\varepsilon(u_0-W_{\delta_{\varepsilon}(t),\xi})\nonumber\\
&\quad\quad\quad\quad\quad\quad\quad\quad- f_\varepsilon^\prime (u_0-W_{\delta_{\varepsilon}(t),\xi})\phi_{\delta_{\varepsilon}(t),\xi}\big)Z_idV\bigg)\nonumber\\
&+DJ_\varepsilon (u_0 -W_{\delta_{\varepsilon}(t),\xi}+\phi_{\delta_{\varepsilon}(t),\xi})[\dfrac{\partial \phi_{\delta_{\varepsilon}(t),\exp_\xi (y)}}{\partial y_i}]|_{y=0}\nonumber\\
&+ O(\left\|R_{\delta_\varepsilon (t),\xi} \right\|_{P_g}\left\|\phi_{\delta_\varepsilon (t),\xi} \right\|_{P_g})\nonumber\\
={}& I_4+I_5+I_6+o(\varepsilon),
\end{align}
where
\begin{align}
I_4= {} & \dfrac{\alpha_n (n-4)}{\delta_{\varepsilon}(t)}\int_M (P_g (Z_i) - f_\varepsilon^\prime (u_0-W_{\delta_{\varepsilon}(t),\xi})Z_i)\phi_{\delta_{\varepsilon}(t),\xi} dV,\nonumber\\
I_5= {} &-\dfrac{\alpha_n (n-4)}{\delta_{\varepsilon}(t)}\int_M (f_\varepsilon(u_0-W_{\delta_{\varepsilon}(t),\xi}+\phi_{\delta_{\varepsilon}(t),\xi})-f_\varepsilon(u_0-W_{\delta_{\varepsilon}(t),\xi})\nonumber\\
&\quad\quad\quad\quad\quad\quad\quad\quad- f_\varepsilon^\prime (u_0-W_{\delta_{\varepsilon}(t),\xi})\phi_{\delta_{\varepsilon}(t),\xi})Z_idV,\nonumber\\
I_6= {} &DJ_\varepsilon (u_0 -W_{\delta_{\varepsilon}(t),\xi}+\phi_{\delta_{\varepsilon}(t),\xi})[\dfrac{\partial \phi_{\delta_{\varepsilon}(t),\exp_\xi (y)}}{\partial y_i}]|_{y=0}.\nonumber
\end{align}
We begin by estimating the terms $I_3$ and $I_6$. We recall that
$$DJ_{\varepsilon} (u_0-W_{\delta_{\varepsilon}(t),\xi}+\phi_{\delta_{\varepsilon}(t),\xi})[.]=\sum_{i=0}^n \lambda_{i} \left\langle Z_{i}, .\right\rangle_{P_g}.$$
Arguing the same way as in Proposition \ref{propfin}, we have
$$DJ(u_0 -W_{\delta_{\varepsilon}(t),\xi}+\phi_{\delta_{\varepsilon}(t),\xi})[\dfrac{\partial \phi_{\delta_{\varepsilon}(t),\xi}}{\partial t}]= 
O\left(\left\|\phi_{\delta_{\varepsilon}(t),\xi} \right\|_{L^{\frac{2n}{n+4}}}\sum_{i=0}^n |\lambda_i|\right),$$
and
$$DJ_\varepsilon (u_0 -W_{\delta_{\varepsilon}(t),\xi}+\phi_{\delta_{\varepsilon}(t),\xi})[\dfrac{\partial \phi_{\delta_{\varepsilon}(t),\exp_\xi (y)}}{\partial y_i}]|_{y=0}=
O\left(\dfrac{\left\|\phi_{\delta_{\varepsilon}(t),\xi} \right\|_{L^{\frac{2n}{n+4}}}\sum_{i=0}^n |\lambda_i|}{\delta_\varepsilon(t)}\right).$$
We claim that $|\lambda_i| =O(\varepsilon \ln \varepsilon)$, for all $i=0,\ldots ,n$. Using \eqref{ortho}, to prove the claim, we just need to show that $DJ(u_0 -W_{\delta_{\varepsilon}(t),\xi}+\phi_{\delta_{\varepsilon}(t),\xi})[Z_i ]=O(\varepsilon \ln \varepsilon)$, for all $i=0,\ldots ,n$. Since $\phi_{\delta_{\varepsilon}(t),\xi} \in K_{\delta_{\varepsilon}(t),\xi}^\bot$, 
using H\"older inequality, \eqref{eqprop4.1}, Lemma \ref{reste} and rough estimates, we have
\begin{align*}
\lefteqn{DJ_\varepsilon(u_0-W_{\delta_{\varepsilon}(t),\xi}+\phi_{\delta_{\varepsilon}(t),\xi})[Z^i]}{} &\\
&= \int_M P_g(u_0-W_{\delta_{\varepsilon}(t),\xi}) Z_i dV -\int_M f_\varepsilon(u_0-W_{\delta_{\varepsilon}(t),\xi}+\phi_{\delta_{\varepsilon}(t),\xi})Z_idV\\
&=  \int_M (P_g(u_0-W_{\delta_{\varepsilon}(t),\xi}) - f_\varepsilon(u_0-W_{\delta_{\varepsilon}(t),\xi}))Z_i dV\\
&\quad-\int_M (f_\varepsilon(u_0-W_{\delta_{\varepsilon}(t),\xi}+\phi_{\delta_{\varepsilon}(t),\xi}) - f_\varepsilon(u_0-W_{\delta_{\varepsilon}(t),\xi}))Z_idV\\
&\leq \left\|P_g(u_0-W_{\delta_{\varepsilon}(t),\xi})-f_\varepsilon(u_0-W_{\delta_{\varepsilon}(t),\xi})\right\|_{L^{\frac{2n}{n+4}}} \left\|Z_i\right\|_{L^{2^\ast}}\\
&\quad+ \left\|f_\varepsilon(u_0-W_{\delta_{\varepsilon}(t),\xi}+\phi_{\delta_{\varepsilon}(t),\xi})-f_\varepsilon(u_0-W_{\delta_{\varepsilon}(t),\xi})\right\|_{L^{\frac{2n}{n+4}}} \left\|Z_i\right\|_{L^{2^\ast}}\\
&\leq O(\left\|P_g(u_0-W_{\delta_{\varepsilon}(t),\xi})-f_\varepsilon(u_0-W_{\delta_{\varepsilon}(t),\xi})\right\|_{L^{\frac{2n}{n+4}}})\\
&\quad+O( \left\|\phi_{\delta_{\varepsilon}(t),\xi} \right\|_{L^{\frac{2n}{n-4}}}(\left\|W_{\delta_{\varepsilon}(t),\xi}\right\|_{L^{\frac{2n}{n-4}}}^{2^\ast-2-\varepsilon}+\left\|\phi_{\delta_{\varepsilon}(t),\xi}\right\|_{L^{\frac{2n}{n-4}}}^{2^\ast -2-\varepsilon} ))\\
&\leq  O(\varepsilon \ln \varepsilon ).
\end{align*}

Combining the previous estimates, we get
\begin{eqnarray}
\label{1505e3}
DJ_\varepsilon(u_0 -W_{\delta_{\varepsilon}(t),\xi}+\phi_{\delta_{\varepsilon}(t),\xi})\left[\dfrac{\partial \phi_{\delta_{\varepsilon}(t),\xi}}{\partial t}\right]=O(\varepsilon^2 (\ln \varepsilon)^2),
\end{eqnarray}
and 
\begin{equation}
\label{1505e4}
DJ_\varepsilon (u_0 -W_{\delta_{\varepsilon}(t),\xi}+\phi_{\delta_{\varepsilon}(t),\xi})\left[\dfrac{\partial \phi_{\delta_{\varepsilon}(t),\exp_\xi (y)}}{\partial y_i}\right]|_{y=0}=O(\varepsilon^{\frac{3}{2}} (\ln \varepsilon)^2).
\end{equation}
Now let us estimate $I_2$ and $I_5$. Noticing that, if $8\leq n\leq 13$,
 $$\left\|(u_0-W_{\delta_{\varepsilon}(t),\xi})^{2^\ast -3-\varepsilon}Z_i\right\|_{L^{\frac{n}{4}}}=O(\varepsilon^{-\frac{1}{4}}),$$ we obtain, using \eqref{dp1}, for $i=0,\ldots ,n$,
\begin{align}
\label{1505e5}
\lefteqn{\int_M (f_\varepsilon(u_0-W_{\delta_{\varepsilon}(t),\xi}+\phi_{\delta_{\varepsilon}(t),\xi})-f_\varepsilon(u_0-W_{\delta_{\varepsilon}(t),\xi})-f_\varepsilon^\prime (u_0-W_{\delta_{\varepsilon}(t),\xi})\phi_{\delta_{\varepsilon}(t),\xi})Z_i dV}{}&\nonumber\\
&\leq  C \left\{
\begin{array}{lr}\displaystyle\int_M (u_0-W_{\delta_{\varepsilon}(t),\xi})^{2^\ast -3-\varepsilon}\phi_{\delta_{\varepsilon}(t),\xi}^2 Z_idV\ &\mathrm{if}\ 12\leq n\leq 13, \\
\displaystyle\int_M ((u_0-W_{\delta_{\varepsilon}(t),\xi})^{2^\ast -3-\varepsilon}\phi_{\delta_{\varepsilon}(t),\xi}^2+\phi_{\delta_{\varepsilon}(t),\xi}^{2^\ast-1-\varepsilon} Z_idV)\ &\mathrm{if}\ 8\leq n< 12, 
\end{array}\right. \nonumber\\
&\leq  C \left\{
\begin{array}{lr}\left\|(u_0-W_{\delta_{\varepsilon}(t),\xi})^{2^\ast -3-\varepsilon}Z_i\right\|_{L^{\frac{n}{4}}}\left\|\phi_{\delta_{\varepsilon}(t),\xi}\right\|_{L^{\frac{2n}{n-4}}}^2 \ &\mathrm{if}\ 12\leq n\leq 13, \\
\left\|\phi_{\delta_{\varepsilon}(t),\xi}\right\|_{L^{\frac{2n}{n-4}}}^2\left\|(u_0-W_{\delta_{\varepsilon}(t),\xi})^{2^\ast -3-\varepsilon}Z_i\right\|_{L^{\frac{n}{4}}}&\\ 
\quad+\left\|Z_i\right\|_{L^{\frac{2n}{n-4}}}\left\|\phi_{\delta_{\varepsilon}(t),\xi}\right\|_{L^{\frac{2n}{n-4}}}^{2^\ast-1-\varepsilon} & \mathrm{if}\ 8\leq n< 12, 
 \end{array}\right. \nonumber\\
&\leq  O(\varepsilon^{2-\frac{1}{4}}(\ln \varepsilon)^2))\ when\ 8\leq n\leq 13.
\end{align}
Finally, let us estimate $I_1$ and $I_4$. Since $\left\|P_g (Z_i)-f_\varepsilon^\prime (W_{\delta_{\varepsilon}(t),\xi})Z_i\right\|_{L^{\frac{2n}{n+4}}}=O(\varepsilon \ln \varepsilon)$ (see \cite{MR2859126}, inequality (4.17)) and since, using rough estimates, 
$$\left\|u_0^{2^\ast -2-\varepsilon}Z_i \right\|_{L^{\frac{2n}{n+4}}}+\left\|W_{\delta_{\varepsilon}(t),\xi}^{2^\ast -3-\varepsilon}Z_i\right\|_{L^{\frac{2n}{n+4}}}=O(\varepsilon \ln \varepsilon),$$  
we obtain
\begin{align}
\label{1505e6}
\lefteqn{\int_M (P_g (Z_i) - f_\varepsilon^\prime (u_0-W_{\delta_{\varepsilon}(t),\xi})Z_i)\phi_{\delta_{\varepsilon}(t),\xi} dV} {}\nonumber\\
 {} &\leq  C\left( \left\|P_g (Z_i)-f_\varepsilon^\prime (W_{\delta_{\varepsilon}(t),\xi})Z_i\right\|_{L^{\frac{2n}{n+4}}}\right.\nonumber\\
  &\quad\quad\quad+\left. \left\|(f_\varepsilon^\prime (u_0-W_{\delta_{\varepsilon}(t),\xi})-f_\varepsilon^\prime (W_{\delta_{\varepsilon}(t),\xi}))Z_i \right\|_{L^{\frac{2n}{n+4}}}\right)\left\|\phi_{\delta_{\varepsilon}(t),\xi} \right\|_{L^{2^\ast}}\nonumber\\
{} &\leq  C\varepsilon \ln \varepsilon (\varepsilon \ln \varepsilon + \left\|u_0^{2^\ast -2-\varepsilon}Z_i \right\|_{L^{\frac{2n}{n+4}}}+\left\|W_{\delta_{\varepsilon}(t),\xi}^{2^\ast -3-\varepsilon}Z_i\right\|_{L^{\frac{2n}{n+4}}})\nonumber\\
{} &\leq  O(\varepsilon^2 \ln \varepsilon^2).
\end{align}
The lemma now follows from \eqref{1505e1}, \eqref{1505e2}, \eqref{1505e3}, \eqref{1505e4}, \eqref{1505e5} and \eqref{1505e6}. 

 \end{proof}

\bibliographystyle{plain}
\bibliography{paneitzbib}
\end{document}